\documentclass[12pt]{amsart}

\usepackage{pdfsync}         
\usepackage{enumerate}
\usepackage{amssymb}
\usepackage{subcaption}
\usepackage{verbatim}

\usepackage[toc,page]{appendix}

\usepackage[OT1]{fontenc}    %
\usepackage{graphicx}

\usepackage{palatino}
\usepackage[left=3cm,top=3.8cm,right=3cm]{geometry}        

\usepackage{xcolor}
\usepackage[pagebackref]{hyperref}
\hypersetup{
   colorlinks,
    linkcolor={red!60!black},
    citecolor={blue!60!black},
    urlcolor={blue!90!black}
}

\numberwithin{equation}{section}

\theoremstyle{plain}
\newtheorem{theorem}{Theorem}[section]
\newtheorem{corollary}[theorem]{Corollary}
\newtheorem{prop}[theorem]{Proposition}
\newtheorem{lemma}[theorem]{Lemma}

\theoremstyle{remark}
\newtheorem{remark}[theorem]{Remark}

\theoremstyle{definition}
\newtheorem{definition}[theorem]{Definition}

\newcommand{\be}{\begin{equation}}
\newcommand{\ee}{\end{equation}}

\newcommand{\1}{\mathbf{1}}

\newcommand{\e}{\varepsilon}
\newcommand{\N}{\mathbb{N}}
\newcommand{\R}{\mathbb{R}}

\newcommand{\dist}{\mathrm{dist}}

\newcommand{\cN}{\mathcal{N}}

\newcommand{\al}{\alpha}
\newcommand{\de}{\delta}
\newcommand{\om}{\omega}
\newcommand{\Om}{\Omega}

\def\rr{{\mathbb R}}
\def\su{\subset}

\def\al{\alpha}
\def\ga{\gamma}

\def\de{\delta}

\def\Om{\Omega}
\def\om{\omega}

\def\ep{\varepsilon}

\def\dim{{\rm dim}\, }
\def\dist{{\rm dist}\, }

\def\lkb{\lesssim}
\def\gkb{\gtrsim}

\def\phi{\varphi}

\DeclareMathOperator{\hdim}{dim_H}

\DeclareMathOperator{\diam}{diam}

\newcommand{\wt}{\widetilde}

\title{An improved bound for the dimension of $(\alpha,2\alpha)$-Furstenberg sets}

\author{Korn\'{e}lia H\'{e}ra}
\address{Department of Mathematics, The University of Chicago, 5734 S. University Avenue,
Chicago, IL 60637, USA}
\email{herakornelia@gmail.com}
\thanks{KH was partially supported by the Hungarian National Research, Development and Innovation Office - NKFIH 124749,
and by the Tempus Public Foundation via the Campus Mundi program}

\author{Pablo Shmerkin}
\address{Department of Mathematics and Statistics, Torcuato Di Tella University, and CONICET, Buenos Aires, Argentina}
\email{pshmerkin@utdt.edu}
\urladdr{http://www.utdt.edu/profesores/pshmerkin}
\thanks{PS has received funding from a University of St Andrews Global Fellowship and from the European Research Council (ERC) under the European Union's Horizon 2020 research and innovation programme (grant agreement No. 803711)}

\author{Alexia Yavicoli}
\address{School of Mathematics and Statistics, University of St. Andrews, St Andrews, KY16 9SS, UK}
\email{ay41@st-andrews.ac.uk, \ alexia.yavicoli@gmail.com}
\thanks{AY is supported by the Swiss National Science Foundation, grant n$^{\circ}$ P2SKP2\_184047.}

\subjclass[2010]{Primary: 28A78, 05B30}
\keywords{Hausdorff dimension, Furstenberg sets, discretized sets}

\begin{document}

\begin{abstract}
We show that given $\alpha \in (0, 1)$ there is a constant $c=c(\alpha) > 0$  such that any planar $(\alpha, 2\alpha)$-Furstenberg set has Hausdorff dimension at least $2\alpha + c$. This improves several previous bounds, in particular extending a result of Katz-Tao and Bourgain. We follow the Katz-Tao approach with suitable changes, along the way clarifying, simplifying and/or quantifying many of the steps.
\end{abstract}

\maketitle

\tableofcontents

\section{Introduction}

Given $\alpha\in (0,1]$, we say that a set $E\subset \R^2$ is an $\alpha$-Furstenberg set  if for every direction $\omega \in S^1$ there is a line $L_\omega$ in direction $\omega$ such that $\hdim(E \cap L_\omega)\geq \alpha$. In \cite{Wolff99}, T. Wolff introduced the problem of estimating
\[
\gamma(\alpha):=\inf \{\hdim(E) : \ E \text{ is an $\alpha$-Furstenberg set}\},
\]
where $\hdim$ stands for Hausdorff dimension. This is a variant of the well-known Kakeya problem, in which one seeks full line segments instead of sets of dimension $\alpha$ in every direction. The problem of computing $\gamma(\alpha)$ is still wide open. Wolff \cite{Wolff99} showed that
\be \label{eq:Wolff-bound}
\max\left\{2\alpha, \alpha +\frac{1}{2}\right\} \leq \gamma(\alpha) \leq \frac{3\alpha}{2}+\frac{1}{2}.
\ee
He also conjectured that the real value is given by the upper bound, that is, $\gamma(\alpha)=\frac{3\alpha}{2}+\frac{1}{2}$.

When $\alpha=\frac{1}{2}$ both lower bounds are equal to $1$, which makes the value somewhat special. In \cite{KatzTao01}, N.~Katz and T.~Tao asked whether $\gamma(1/2)\ge 1+\e$ for some absolute $\e>0$. While they didn't answer this question, they connected it to two other well-known problems: the Erd\H{o}s-Volkmann ring problem and the Falconer distance set problem. To be more precise, Katz and Tao introduced discretized versions of these three problems, proved that the discretized versions are equivalent to each other, and that the discretized version of the $1/2$-Furstenberg problem implies that $\gamma(1/2)\ge 1+\e$.

Not too long after, J. Bourgain \cite{Bourgain03} proved the $\delta$-discretized version of the Erd\H{o}s-Volkmann ring problem (which is now known as the discretized sum-product theorem). Together with the results from \cite{KatzTao01}, this yields the unconditional bound $\gamma(1/2)\ge 1+\e$. The value of $\e$, although effective in principle, is very small. This has been the only improvement over the bounds in \eqref{eq:Wolff-bound}, although we should mention that T. Orponen \cite{Orponen17} obtained an $\e$-improvement on the \emph{packing} dimension of Furstenberg sets in the range $\alpha\in (1/2,1)$.

U. Molter and E. Rela \cite{MolterRela12} generalized the notion of Furstenberg sets as follows: given $\alpha\in (0,1]$ and $\beta\in (0,1]$, we say that a set $E\subset\R^2$ is in the class $F_{\alpha,\beta}$ if there exists a set $\Omega\subset S^1$ of directions with $\hdim(\Omega)\ge \beta$, such that for all $\omega$ there is a line $L_\omega$ in direction $\omega$ with $\hdim(E\cap L_\om)\ge\alpha$. In other words, they consider a fractal set of directions, rather than every direction as in the original problem. By adapting Wolff's method, Molter and Rela generalized Wolff's lower bounds to the class  $F_{\alpha,\beta}$. Write $\gamma(\alpha,\beta)=\inf \left\{\hdim(E): E\in F_{\alpha,\beta}\right\}$. Molter and Rela proved that
\be \label{eq:MolterRela}
\gamma(\alpha,\beta) \ge \max \left\{2\alpha+\beta-1,\alpha+\frac{\beta}{2}\right\}.
\ee
More recently, N.~Lutz and D.~Stull \cite{LutzStull17}, using Kolmogorov complexity methods, made an improvement over this bound in the range $\beta<2\alpha$:
\be \label{eq:LutzStull}
\gamma(\alpha,\beta) \geq \alpha+ \min\{ \beta, \alpha \}
\ee
In the appendix we give a more classical proof of a more general version of this statement, and extend it to higher dimensions, based on an idea we learned from L.~Guth. In dimension $n\ge 3$, the bound \eqref{eq:LutzStull} improves upon those of \cite{HKM19} when $\beta\le 2\alpha$.

We note that the bound from Lutz and Stull is sharp for $\beta\leq\alpha$, as illustrated by a ``Cantor target'' construction: let $A \subset [0,1]$ such that $\hdim(A)=\dim_B(A)=\alpha$, and take $\Omega \subset S^1$ with $\hdim(\Omega)=\beta$. We define $A_{\omega}$ as a rotation of $A$ by angle $\omega$ around the origin, and set $E:=\bigcup_{\omega \in \Omega}A_{\omega}$. Then, by using polar coordinates and \cite[Corollary 7.4]{Falconer}, we see that $\hdim(E)= \hdim(A \times \Omega)=\alpha +\beta$.

These were the best known bounds prior to this article (see, however, \cite{Oberlin14, Venieri17, HKM19, Hera19} for progress on the corresponding problem in higher dimensions). Note that, because of the $\min\{\beta,\alpha\}$ term, the bound \eqref{eq:LutzStull} does not distinguish sets in $F_{\alpha,\alpha}$ from the (intuitively much larger) sets in $F_{\alpha,2\alpha}$. Moreover, both \eqref{eq:MolterRela} and \eqref{eq:LutzStull} yield the same bound $\gamma(\alpha,2\alpha)\ge 2\alpha$. This suggests that it may be of interest to improve upon this bound.

The main result of this paper is an $\e$-improvement, that is, we show that $\gamma(\alpha,2\alpha)\ge 2\alpha+c$, where $c>0$ depends only on $\alpha$. In fact, we prove a more general statement. We are able to consider values of $\beta$ a bit smaller than $2\alpha$, and consider a larger class of sets where, rather than working with lines in different directions in some fractal set, we just work with a family of lines of some dimension (some of these lines may be parallel to each other).
\begin{theorem} \label{thm:main}
Given $\alpha\in (0,\frac{1}{2}]$ there is $c=c(\alpha)>0$ (depending continuously on $\alpha$) such that the following holds.

Let $E\subset\R^2$ be a set with the following property: there is a set $\Omega$ of lines with $\hdim(\Omega)\ge 2\alpha$, such that
\[
\hdim(E\cap \om) \ge \alpha \quad\text{for all } \om\in\Om.
\]
Then
\[
\hdim(E) \ge 2\alpha +c
\]
In particular, $\gamma(\alpha,2\alpha)\ge 2\alpha+c$.
\end{theorem}
No measurability is required. Note that the set $A(2,1)$ of lines in $\R^2$ is a two-dimensional manifold, and hence it makes sense to speak of the Hausdorff dimension of a set of lines. Notice also that the theorem implies, as a particular case, that there is $c>0$ such that classical $(1/2-c)$-Furstenberg sets have dimension $\ge 1+c$. While this statement can be tracked down from the proofs of \cite{KatzTao01}, to our knowledge it hadn't been explicitly pointed out before.

To prove the theorem, we follow many of the ideas of Katz and Tao in \cite{KatzTao01}, but we simplify, clarify, adapt and quantify many of the steps. As explained above, Katz and Tao reduce the proof of the bound $\gamma(1/2)\ge 1+\e$ to a discretized statement, and then reduce the proof of the discretized statement to (what is now called) the discretized sum-product theorem. We also reduce the proof of Theorem \ref{thm:main} to a (different) discretized analog, but we do it in a different way which we believe is more straightforward. To prove the discretized statement, we follow the main ideas of Katz and Tao's approach, although the details differ at most places.  In the end, we rely on Bourgain's discretized projection theorem \cite{Bourgain10} rather than sum-product estimates, which allows us to make the proof shorter. This is not surprising since the projection theorem is a refinement of the sum-product theorem (in fact, many of the steps in going from sum-product to projection are implicit in \cite{KatzTao01, Bourgain03}). While focusing on the simplicity of the arguments rather than optimization, we track the quantitative dependence of $c$ on the parameters from Bourgain's projection theorem, see Remark \ref{rem:quantitative}.

\section{Definitions and main tools}

\subsection{Notation}

We denote by $|\cdot|$ both Lebesgue measure (for ``large'' subsets of $\R^n$, usually unions of balls) and cardinality (for finite sets). The meaning should always be clear from context.

We denote the open ball in $\R^n$ with centre $x$ and radius $r$ by $B^n(x,r)$. We usually skip the superindex $n$, when it is clear from context. We also use the notation $B_r$ to denote an arbitrary ball of radius $r$.

In what follows we will work with a small parameter $\delta$. We use the notation $A\lesssim B$ for $A\le C (\log(\delta^{-1}))^C B$ where $C$ is a constant that depends only on the ambient space, and may change from line to line. Likewise, we write $A\gtrsim B$ for $B\lesssim A$, and $A\approx B$ for $A\lesssim B\lesssim A$.

The open $r$-neighborhood of the set $A$ will be denoted by $A^{(r)}=\{ x: \dist(x,A)<r\}$.

\subsection{Discretized sets}
\label{subsec:discretized-sets}

We will often work with $\delta$-discretized sets:
\begin{definition}
We say that $A\subset\R^n$ is a $\delta$-discretized set if it is a union of $\delta$-balls.
\end{definition}

The following lemma collects some basic facts about discretized sets; they will be used without further reference in the rest of the paper.
\begin{lemma}
There is $C_n>0$ such that the following hold. Let $A\subset\R^n$ be $\delta$-discretized.
\begin{enumerate}
  \item There are $\delta$-discretized sets $A_*, A^*$ such that $A_*$ is a union of \emph{disjoint} $\delta$-balls, $A^*$ is the union of $\delta$-balls with overlapping bounded by $C_n$, $|A^*|\le C_n |A_*|$, and $A_*\subset A\subset A^*$.
  \item $|A^{(\delta)}|\le C_n |A|$.
\end{enumerate}
\end{lemma}

Among $\delta$-discretized sets, we will often deal with a special family of sets that, in a sense, ``look like a set of dimension $\alpha$ at scale $\delta$''.
\begin{definition}
We say that $A\subset\R^n$ is a $(\delta,\alpha,\e)$-set if the following conditions hold:
\begin{itemize}
  \item $A$ is a $\delta$-discretized subset of $B^n(0,2)$.
  \item For all $x$ and all $r\in [\delta,2]$ it holds that
\[
|A\cap B^n(x,r)|\lesssim  \delta^{n-\e} (r/\delta)^\alpha \quad \text{(non-concentration hypothesis).}
\]
\item $|A|\gtrsim \delta^{n-\alpha+\e}$.
\end{itemize}
In the case $\e=0$, we simply say that $A$ is a $(\delta,\alpha)$-set. When we want to emphasize the ambient dimension, we will write $(\delta,\alpha,\e)_n$-set.
\end{definition}
Note that applying the non-concentration hypothesis with $r=2$, we get that a $(\delta,\alpha)_n$-set has measure $\approx \delta^{n-\alpha}$.

The following lemma extracts a $(\delta,s,\eta)$-set from a given $\delta$-discretized set. It is essentially \cite[Refinement 2.2]{KatzTao01}, but we give the details of the proof for completeness.
\begin{lemma} \label{lem:refinement}
Let $E\subset B^n(0,2)$ be a $\delta$-discretized set with $|E|\lesssim \delta^{n-s}$. Then for every $\eta>0$ there exist sets $E^*,E^{**}$ such that:
\begin{enumerate}
  \item $E\subset  E^* \cup E^{**} \subset E^{(\delta)}$,
  \item $E^*\cap B(0,2)$ is a $(\delta,s,\eta)_n$-set,
  \item $E^{**}= \cup_{\delta'} E^{**}_{\delta'}$, where $\delta'$ ranges over dyadic numbers in $[2\delta,2]$, and $E^{**}_{\delta'}$ can be covered by $\lesssim \delta^\eta (\delta')^{-s}$ balls of radius $\delta'$.
\end{enumerate}
\end{lemma}
\begin{proof}
For every dyadic number $\delta'$, we define \[E^{**}_{\delta'}:=\{x \in E^{(\delta)}: \ |E^{(\delta)} \cap B(x,\delta')|\geq \delta^{n-\eta} (\tfrac{\delta'}{\delta})^s \},\]
\[E^{**}:=\bigcup_{2\delta\leq \delta'\leq 2} E^{**}_{\delta'},\]
and \[E^*:=(E \setminus E^{**})^{(\delta)}=\bigcup_{x \in E \setminus E^{**}}B(x, \delta) \subseteq E^{(\delta)}.\]

\begin{enumerate}
\item This is clear from the definitions.
\item The set $E^*$ is $\delta$-discretized by definition. It is enough to check the non-concentration assumption for $x \in E^*$ and dyadic $r\in [\delta,2]$.

If $y \in E\setminus E^{**}$, then $y \notin E_r^{**}$ for every $2 \delta \leq r \leq 2$, so $|E^{(\delta)}\cap B(y,r)| \lesssim \delta^n (\frac{r}{\delta})^s$. Since $E^* \subseteq E^{(\delta)}$, we have $|E^*\cap B(y,r)| \lesssim \delta^n (\frac{r}{\delta})^s$. In general, if $y \in E^*$, there exists $y' \in E\setminus E^{**}$ with $|y-y'|<\delta$ and $y'\in E\setminus E^{**}$. Then, for every dyadic $r \in [\delta, 1]$, \[
|E^* \cap B(y,r)| \leq |E^* \cap B(y',2r)| \lesssim \delta^n (\frac{r}{\delta})^s.
\]
If $r=2$ we cover $B(y,r)$ by $C_n$ balls of radius $1$ and go back to the previous case.

\item By the $5r$-covering theorem, there is a disjoint collection of balls $\{B(x_i,\delta'/5)\}_{i=1}^M$ centered in $E''_{\delta'}$, such that $E''_\delta\subset \cup_{i=1}^M B(x_i,\delta')$. In particular, each $x \in E''_{\delta'}$  belongs to at most $c_n$ of the balls $B(x_i,\delta')$.  Then, using the hypothesis $|E|\gtrsim \delta^{n-s}$, and that $|E|\approx |E^{(\delta)}|\ge |E^{**}_{\delta'}|$, we get
\begin{align*}
\delta^{n-s} &\gtrsim c_n |E| \gtrsim c_n |E^{**}_{\delta'}|\\
&\gtrsim \int \sum_{i=1}^{M} \1_{B(x_j, \delta') \cap E^{**}_{\delta'}} \ge M \delta^{n-\eta} (\tfrac{\delta'}{\delta})^s.
\end{align*}
Hence $M \lesssim \delta^{\eta} (\delta')^{-s}$, as claimed.
\end{enumerate}
\end{proof}

We recall the definition of (spherical) Hausdorff content of a subset $A$ of $\R^d$:
\[
\mathcal{H}_\infty^\alpha(A):=\inf \left\{\sum_i r_i^{\alpha}:  \text{ there is a cover of } A \text{ with balls of radii } r_i>0 \right\}.
\]
Hausdorff content is countably subadditive but (unlike Hausdorff measure) is not additive on Borel sets.

\subsection{Metric and measure on the space of lines}
\label{sec: metriclines}

Let $A(n,1)$ be the manifold of affine lines in $\R^n$, and let $G(n,1)\subset A(n,1)$ be the projective space of lines through the origin. Since we will be working with subsets of $A(n,1)$, we extend the definitions from \S\ref{subsec:discretized-sets} to this setting. For this, we need to fix a metric and a measure on $A(n,1)$. We follow \cite[\S 3.16]{Mattila95}. Given two lines $\ell_1,\ell_2\in A(n,1)$ we can write $\ell_i =\langle e_i\rangle+v_i$ where $e_i\in S^{n-1}$ and $v_i\in e_i^{\perp}$. We then define
\[
d(\ell_1,\ell_2) = |e_1-e_2|+ |v_1-v_2|.
\]
Note that, up to a multiplicative constant, $|e_1-e_2|$ equals the angle between the lines through the origin parallel to $\ell_1$ and $\ell_2$.

For lines that intersect the ball $B^n(0,2)$ (the context we will usually be working on), up to a constant factor in the radius, the ball $B(\ell,r)$ is given by all the lines $\ell'$ such that $\ell'\cap B^n(0,3)\subset \ell^{(r)}$. More precisely, there is a constant $C_n>0$ such that, for $r\in (0,2]$,
\[
B(\ell,r/C_n) \subset \{ \ell'\in A(n,1): \ell'\cap B^n(0,3)\subset \ell^{(r)}\} \subset B(\ell,C_n r).
\]
We will need to use an explicit formula for the distance in the following parametrization of lines in the plane that avoid the origin. Let $\ell_{v}=\{ x\in\R^2:x\cdot v=1\}$ for $v\in\R^2\setminus\{0\}$.
\begin{lemma} \label{lem:distance-lines}
\[
d(\ell_v,\ell_{v'}) = \left|\frac{v}{|v|}-\frac{v'}{|v'|}\right| + \left|\frac{v}{|v|^2}-\frac{v'}{|v'|^2}\right|.
\]
In particular,
\[
\frac{\min\{1,|v|\}}{|v||v'|} \le \frac{d(\ell_v,\ell_{v'})}{|v-v'|} \le \left(\frac{4}{|v|^2}+\frac{1}{|v||v'|}\right) .
\]
\end{lemma}
\begin{proof}
The first claim is a direct calculation using that $v\in \ell_v^{\perp}$. A little algebra yields the right-hand side inequality in the second claim. For the left-hand side inequality, write $e(v)=v/|v|$ and note that, applying the triangle inequality with intermediate vector $|v'|e(v)$,
\[
|v-v'|\le |v'||e(v)-e(v')|+ \big||v|-|v'|\big| .
\]
Also by the triangle inequality,
\[
\left| \frac{v}{|v|^2}-\frac{v'}{|v'|^2}\right| \ge \left|\frac{1}{|v|}-\frac{1}{|v'|} \right| = \frac{\big||v|-|v'|\big|}{|v||v'|}
\]
Thus
\[
d(\ell_v,\ell_{v'})\ge \frac{1}{|v||v'|} \left( |v||v'||e(v)-e(v')|+ \big||v|-|v'|\big|\right)\ge \frac{\min\{1,|v|\}}{|v||v'|} |v-v'|,
\]
as claimed.
\end{proof}

We now define a measure on $A(n,1)$. Firstly, there is a measure $\rho_n$ on $G(n,1)$ defined by identifying lines with the points they intersect in the upper half-sphere. This causes trouble for lines lying in the horizontal hyperplane, but they form a set of measure zero; otherwise, we can follow \cite[\S 3.9]{Mattila95} and define $\rho_n$ as the only probability measure on $G(n,1)$ invariant under the action of the orthogonal group; the resulting measures are the same up to a multiplicative constant. Now we define a measure on $A(n,1)$ via
\[
\widehat{\rho}_n(L) = \int \mathcal{H}^{n-1}\{ v\in \ell^\perp: \ell+v\in L\} \,d\rho_n(\ell).
\]
It is easy to see that there is a constant $C_n>0$ such that
\[
C_n^{-1} r^{2n-2} \le \widehat{\rho}_n(B(\ell,r)) \le C_n r^{2n-2} \quad (\ell\in A(n,1)),
\]
which agrees with the fact that $A(n,1)$ is a $(2n-2)$-dimensional manifold. We will abuse notation and denote the measure $\widehat{\rho}_n$ on $A(n,1)$ also by $|\cdot|$.

We extend the notion of $\delta$-discretized and $(\delta,\alpha,\e)$-set to subsets of $A(n,1)$. We always assume that the underlying metric and measure are $d$ and $\widehat{\rho}_n$ defined above (with the latter denoted $|\cdot|$). As is natural, in the definition of $(\delta,\alpha,\e)$-set, we use the dimension $2n-2$ in place of $n$.

\subsection{Main tools}

In this section we introduce the tools we will use in the proof of Theorem \ref{thm:main}.  Frostman's Lemma states that given a Borel set $A\subset\R^n$ with $\mathcal{H}_\infty^\alpha(A)>0$, there exists a Borel probability measure $\mu$ with topological support contained in $A$, such that
\[
\mu(B(x,r))\le C\, r^\alpha \quad\text{for all }x\in\R^n, r>0.
\]
Here $C$ is a constant depending only on $d$ and $\mathcal{H}_\infty^\alpha(A)$. We will need the following discretized version of Frostman's Lemma, due to K.~F\"{a}ssler and T.~Orponen \cite[Proposition A.1]{FasslerOrponen14} (see also \cite[Definition 2.12]{FasslerOrponen14}); they state it only in $\R^3$ but the proof works without changes in any dimension.

\begin{lemma} \label{lem:discretization}
Let $\delta>0$, and let $A\subset B^n(0,2)$ be a set such that $\mathcal{H}_\infty^\alpha(A)>0$. Then there exists a $\delta$-discretized set $A^*\subset A^{(\delta)}$ such that
\[
|A^*\cap B^n(x,r)|\le  \delta^{n} (r/\delta)^\alpha \quad\text{for all }r\in [\delta,2],
\]
and $|A^*|\ge c_n \mathcal{H}_\infty^\alpha(A) \delta^{n-\alpha}$, where $c_n>0$ depends only on the ambient dimension.

In particular, if $\mathcal{H}_\infty^\alpha(A)\gtrsim 1$, then $A^*$ is a $(\delta,\alpha)$-set.
\end{lemma}

\begin{remark} \label{rem:frostman-lines}
Lemma  \ref{lem:discretization} also holds in $A(n,1)$ (with a different constant). Indeed, we can introduce coordinates that make $(A(n,1),d)$ locally bi-Lipschitz to $\R^{2n-2}$. For example, we can identify $A(n,1)$ with $S^{n-1}\times \R^{n-1}$ via $(e,v)\mapsto \{ t e+ \wt{v}_e\}$, where $\wt{v}_e$ is the vector on $e^{\perp}$ that is obtained by rotating $(v_1,\ldots,v_{n-1},0)\in \R^n$ onto $e^\perp$ (in a manner smooth in $e$). Hence, the ball $B(0,2)\subset A(n,1)$ can be covered by $M$ patches (in fact, we can take $M=2$) on which there is a bi-Lipschitz embedding into $\R^{2n-2}$. Given a set $B\in A(n,1)$, by subadditivity of the Hausdorff content we can find one of the patches $P$ such that $\mathcal{H}^\alpha(B\cap P)\ge \mathcal{H}^\alpha(B)/M$, and then apply the Euclidean version to $B\cap P$ going back and forth with the bi-Lipschitz embedding.
\end{remark}

As explained in the introduction, the main tool in our proof is Bourgain's discretized projection theorem from \cite{Bourgain10}. The statement below is a slightly simplification of the original, due to W. He \cite[Theorem 1]{He18}. We only state the case $n=2$, $m=1$, and identify the Grassmanian $G(2,1)$ of lines in $\R^2$ with a subset of the circle. Let $\cN_\delta(X)$ be the $\delta$-covering number of $X$, that is, the smallest number of balls of radius $\delta$ needed to cover $X$.
\begin{theorem} \label{thm:bourgain}
Given $0<\beta<2$ and $\kappa>0$, there is $\lambda>0$  (depending continuously on $\beta,\kappa$) such that the following hold if $\delta$ is sufficiently small (depending on all previous parameters).

Let $F\subset B^2(0,1)$ and let $\mu$ be a probability measure on $S^1$, such that the following conditions hold:
\begin{enumerate}
  \item $\cN_\delta(F) \ge \delta^{\lambda-\beta}$.
  \item $\cN_\delta(F\cap B(x,r)) \le \delta^{-\lambda}r^{\kappa}\cN_\delta(F)$ for all $r\in [\delta,1]$, $x\in B^2(0,1)$.
  \item $\mu(B(e,r)) \le \delta^{-\lambda} r^\kappa$ for all $r\in [\delta,1]$, $e\in S^1$.
\end{enumerate}
Then there is a set $D\subset S^1$ with $\mu(D)\ge 1-\delta^\lambda$ such that if $F'\subset F$ satisfies
\[
\cN_\delta(F')\ge \delta^{\lambda}\cN_\delta(F),
\]
then
\[
\cN_\delta(P_e F') \ge \delta^{-\beta/2-\lambda},
\]
where $P_e(x)=e\cdot x$ is orthogonal projection in direction $e$.
\end{theorem}
Roughly speaking, this theorem says that if $F$ is the union of $\approx \delta^{-\beta}$ balls of radius $\delta$, and satisfies a mild non-concentration assumption (where the exponent can be smaller than $\beta$) then the box-counting number of $P_e F'$ at scale $\delta$ is at least $\delta^{-\beta/2-\lambda}$ for all subsets $F'$ of $F$ satisfying $|F'| \ge \delta^{\lambda}|F|$, for all $e$ outside of a very sparse set of possible exceptions. It is crucial for us that the estimate works for all large subsets $F'$ simultaneously. The $\delta^{-\lambda}$ factor in the second and third assumptions says that no decay is required for large scales (those larger than $\delta^\lambda$), this will be key for us as well.

We note that the fact that $\lambda$ can be taken continuous is not explicitly stated in the literature, but it follows directly from the robustness of the hypotheses and the conclusion of the theorem.

\begin{remark} \label{rem:bourgain}
In our application of Theorem \ref{thm:bourgain}, the set $F$ will not be contained in the unit ball, but it will be contained in a ball of radius $\delta^{-\lambda/4}$ with $\lambda$ small. By a simple scaling argument, applying the theorem to $\delta^{\lambda/4} F$ in place of $F$, we get that the result still holds, except that $\lambda$ has to be replaced by $\lambda/4$, in order to make sure that the first hypothesis holds for the rescaled set $\delta^{\lambda/4} F$.
\end{remark}

\section{Discretization and initial reductions}

\subsection{Definitions}

From now, we will use the following definition of $(\alpha,\beta)$-Furstenberg set:
\begin{definition}[$(\alpha,\beta)$-Furstenberg set]
Given $\alpha\in (0,1]$ and $\beta\in (0,2n-2]$, by an \emph{$(\alpha,\beta)$-Furstenberg set} we mean a subset $E$ of $\R^n$ for which there exists a set of lines $L\subset A(n,1)$ of positive $\beta$-Hausdorff measure such that $\mathcal{H}^\alpha(E\cap\om)>0$ for all $\om\in L$.
\end{definition}
Note that if a set is in the class $F_{\alpha,\beta}$, then it is also an $(\alpha',\beta')$-Furstenberg set for all $\alpha'<\alpha$, $\beta'<\beta$. Also, by the continuity of $c$ in $\alpha$, in order to prove Theorem \ref{thm:main} it is enough to show that the Hausdorff dimension of an $(\alpha,2\alpha)$-Furstenberg set is $\ge 2\alpha+c$.

The next key definition introduces the discretized notion of Furstenberg set we will work with for the rest of the paper.
\begin{definition}
We say that $A\subset B^n(0,2)$ is a \emph{discretized $(\delta,\alpha,\beta)$-Furstenberg set} if $A=\cup_{\omega\in\Omega} R_\omega$, where:
\begin{itemize}
\item The set $\Omega$ is $\delta$-separated and $\Omega^{(\delta)}$ is a $(\delta,\beta)$-set in $A(n,1)$.
\item For each $\omega\in\Omega$, the set $R_\omega$ is a $(\delta,\alpha)_n$-set contained in  $\omega^{(2\delta)}$.
\item $|\Omega|\gtrsim \delta^{-\beta}$
\end{itemize}
\end{definition}

In all the above definitions, we consider $\alpha$ and $\beta$ as constants, and therefore allow the implicit constants $C$ to depend on them.

The next lemma contains our basic discretization estimate.
\begin{lemma} \label{lem:discrete-to-continuous}
Suppose that every discretized $(\delta,\alpha,\beta)$-Furstenberg set has measure $\gtrsim \delta^{n-s}$. Then every $(\alpha,\beta)$-Furstenberg set has Hausdorff dimension at least $s$.
\end{lemma}
\begin{proof}
Assume that every discretized $(\delta,\alpha,\beta)$-Furstenberg set has measure $\gtrsim \delta^{n-s}$, and let $E$ be an $(\alpha,\beta)$-Furstenberg set with line set $L$. There exists $c>0$ such that $\mathcal{H}_\infty^{\beta}(\Omega_1)\ge c$, where
\[
\Omega_1=\Omega_1(c)=\{\om\in L: \mathcal{H}_\infty^{\alpha}(E\cap \om)>c \}.
\]
This is by countable subadditivity of Hausdorff content, and the observation
\[
L = \bigcup_{n} \{\om\in L: \mathcal{H}_\infty^{\alpha}(E\cap \om)>1/n\}.
\]
We take $k_0(c) \in \N$ such that $\sum_{k \geq k_0(c)}\frac{1}{k^2}<c$. Let $\mathcal{C}=\{ B(x_i,r_i)_i\}$ be a cover of $E$ by balls of radius smaller than $2^{-k_0(c)}$.

Let $E_k$ be the union of the $B(x_i,r_i)$ such that $2^{-(k+1)}<r_i\le 2^{-k}$. By countable subadditivity of Hausdorff content and the choice of $k_0(c)$, for each $\omega\in\Omega_1$ there exists $k(\omega)\geq k_0(c)$ such that $\mathcal{H}_\infty^\alpha(E_{k(\omega)}\cap\omega) > k(\omega)^{-2}$.

Again by countable subadditivity of content, there is a fixed value $k_1 \geq k_0(c)$ such that $\mathcal{H}_\infty^\beta(\Omega_2)> k_1^{-2}$, where $\Omega_2=\{\omega \in \Omega_1: k(\omega)=k_1\}$. 

Fix $\delta=2^{-k_1}$, and apply Remark \ref{rem:frostman-lines} to $A=\Omega_2$ and $\beta$ in place of $\alpha$; let $\Omega_3$ be the resulting $(\delta,\beta)$-set. Hence $\Omega_3\subset \Omega_2^{(\delta)}$ and $|\Omega_3|\gtrsim \mathcal{H}^{\beta}_{\infty}(\Omega_2)\delta^{2n-2-\beta}$.

Let $\Omega$ be a maximal $\delta$-separated subset of $\Omega_3$. Note that for each $\omega\in\Omega$, there is $\omega'\in\Omega_2$ such that $d(\omega,\omega')<\delta$. Let $R_\omega$ be the $(\delta,\alpha)$-set obtained from applying Lemma \ref{lem:discretization} to $A= E_{k_1}\cap\omega'$, and $\delta=2^{-k_1}$; note that $\mathcal{H}^{\alpha}_{\infty}(A)>k_1^{-2}$ because $k(\omega')=k_1$ for every $\omega' \in \Omega_2$. Then $R_\omega\subset \omega^{(2\delta)}$, and therefore $E^*:=\cup_{\omega\in\Omega} R_\omega$ is a discretized $(\delta,\alpha,\beta)$-Furstenberg set.

By assumption, $|E^*|\gtrsim \delta^{n-s}$. On the other hand, by construction $E^* \subset E_{k_1}^{(2\delta)}$. Since $E_{k_1}$ is $\delta$-discretized by definition, $|E_{k_1}^{(2\delta)}|\approx |E_{k_1}|$, and therefore $|E_{k_1}| \gtrsim \delta^{n-s}$.

Suppose $E_{k_1}$ is the union of $N$ balls $B(x_i,r_i)$ of radius comparable to $2^{-k_1}$ in the original cover of $E$. Then $N\gtrsim \delta^{-s}$. Now fix $\e>0$. Then
\[
\sum_i\{ r_i^{s-\e}\}\ge \sum\{ r_i^{s-\e}: 2^{-(k_1+1)}<r_i\le 2^{-k_1}\} \ge 2^{-k_1(s-\ep)} N \gtrsim \delta^{-\e}.
\]
But $\delta=2^{-k_1}$ can be made arbitrarily small and hence, by definition of $\gtrsim$, the Hausdorff sum is at least $1$. As the covering was arbitrary, we get that $\hdim(E)\ge s-\e$, which gives the claim since $\e>0$ was arbitrary as well.
\end{proof}

In order to show that the measure of an $(\delta,\alpha,\beta)$-Furstenberg set is large, one needs to control the sizes of the intersections between various of the sets $R_\om$. The next lemma shows that, on the other hand, the products $R_\om\times R_\om$ are nearly disjoint.
\begin{lemma} \label{lem:easy-bound}
Let $E$ be a discretized $(\delta,\alpha,\beta)_n$-Furstenberg set, with associated sets $\Omega$ and $(R_\om)_{\om\in\Omega}$. Then
\[
\left|\bigcup_{\om\in\Om} (R_\om\times R_\om)\right| \approx \sum_{\om\in\Om} |R_\om\times R_\om|.
\]
In particular,
\[
|E\times E| \gtrsim \sum_{\om\in\Omega} |R_\om \times R_\om| \approx \delta^{2n-2\alpha-\beta},
\]
and therefore $|E| \gtrsim \delta^{n-\alpha-\beta/2}$.
\end{lemma}
\begin{proof}
Since $E\times E\supset \cup_{\om\in\Om} (R_\om\times R_\om)$, we only have to show the first claim. Moreover, it is enough to show the $\gtrsim$ direction, since the opposite one is obvious.

Since the $R_\om$ are $(\delta,\alpha)$-sets, by definition there exists a constant $C$ such that
\[
|R_\om\cap B(x,\rho)|\le C \log(1/\delta)^C \delta^{n-\alpha} \rho^{\alpha},\quad |R_\om|\ge C^{-1} \log(1/\delta)^{-C} \delta^{n-\alpha}.
\]
It follows that if $C'_\alpha$ is large enough in terms of $C$ and $\rho=(C'_\alpha)^{-1}\log(1/\de)^{-C'_\alpha}$, then
\[
|R_\om \cap B(x,\rho)|\le |R_\om|/2.
\]
From this and Fubini it follows that
\[
|(R_\om\times R_\om) \setminus\Delta| \ge |R_\om|^2/2,
\]
where
\[
\Delta=\{(x,y) \in R_\om\times R_\om: |x-y| \leq \rho\}.
\]
Since we treat $\alpha$ as a constant, we have $\rho\gtrsim 1$. Let $\wt{\Omega}$ be a maximal $(C''\de/\rho)$-separated subset of $\Omega$, where $C''$ will be chosen momentarily. Since $\Omega$ is $\delta$-separated, we have $|\wt{\Omega}| \gtrsim |\Omega|$. By elementary geometry, if $C''$ is chosen sufficiently large depending only on the ambient dimension, and if $\om\neq\om'\in\wt{\Omega}$, then
\[
\diam(R_\om\cap R_{\om'}) \le \rho/\sqrt{n}.
\]
It follows that if $\om\neq\om'\in \wt{\Omega}$, then
\[
(R_\om\times R_\om)\cap(R_{\om'}\times R_{\om'})\subset\Delta,
\]
and therefore, recalling that $|\wt{\Omega}|\gtrsim |\Omega|\gtrsim \delta^{-\beta}$,
\begin{align*}
\left|\bigcup_{\om\in\Om} (R_\om\times R_\om)\right|  &\ge  \left|\bigcup_{\om\in\wt{\Om}} (R_\om\times R_\om)\setminus\Delta \right| \\
& \ge \sum_{\om\in\wt{\Om}} |R_\om\times R_\om|/2 \\
&\gtrsim \delta^{-\beta} \delta^{2n-2\alpha} \approx \sum_{\om\in\Om} |R_\om\times R_\om|.
\end{align*}
\end{proof}
Incidentally, this estimate together with Lemma \ref{lem:discrete-to-continuous} recovers the lower bound $\alpha+\beta/2$ for the dimension of $(\alpha,\beta)$-Furstenberg sets.

From now on we only deal with the case $n=2$. Let $\gamma=\gamma(\alpha)$ be the supremum of all real numbers such that, if $\de$ is sufficiently small (depending on $\ga$), then every discretized $(\delta,\alpha,2\alpha)_2$-Furstenberg set has measure $\gtrsim \delta^{2-2\alpha-\gamma}$. It follows from Lemma \ref{lem:easy-bound} that $\gamma\ge 0$, and our goal is to show that $\gamma>0$. Our strategy will be to show that if $\gamma$ is very small, this forces a very rigid structure on the discretized Furstenberg set that will ultimately lead to a contradiction with Bourgain's projection theorem.

\section{The proof of Theorem \ref{thm:main}}

\subsection{Strategy}

We summarize the strategy of the proof. In light of Lemma \ref{lem:easy-bound}, if $|E|\approx \delta^{2-2\alpha-\gamma}$ for a very small $\gamma$, this morally means that $E\times E$ is not too different from $\cup_{\om\in\Omega} R_\om\times R_\om$. Using this, we can find a point $y$ with the property that for ``most'' points $x$ in $E$ there exists $\om$ containing both $x$ and $y$. Let $\Omega_y$ be the set of $\omega$ such that $R_\omega$ passes through $y$. Then $\cup_{\om\in \Om_y} R_\om$ (which we recall fills up a big part of $E$) forms a ``fan'' and thus we can count that there must be roughly $\delta^{-\alpha}$ elements in $\Omega_y$. Now fix $\omega_0\in \Omega_y$. Then $\omega_0$ passes very close to $y$; simplifying a little bit, let us assume it passes through $y$. For every $\omega\in \Omega$, let $\Pi_{\omega_0}(\omega)$ be the intersection point $\omega\cap \omega_0$ (this does not exist if $\omega$ and $\omega'$ are parallel, but we ignore this; it does exist most of the time). Because the $\sim \delta^{-\alpha}$ sets $R_{\omega_0}$ cover much of $E$, a simple counting argument shows that ``very often'' the point $\Pi_{\omega_0}(\omega)$ lies in $R_{\omega_0}\cap R_{\omega}\subset E$ (in these arguments it is important that $\beta=2\alpha$). Thus we can see the map $\omega\mapsto \Pi_{\omega_0}(\omega)$ as a sort of projection from $\Omega$, parametrized by $\omega_0\in\Omega_y$, that returns something close to $E\cap \omega_0$. We can then hope to use some projection theorem that tells us that for ``most'' $\omega_0$, the projection $\Pi_{\omega_0}(\Omega)$ is ``large''. If this is the case, then (since $E$ is $\delta$-discretized) $|E\cap \omega_0^{(\delta)}|$ is large for most $\omega_0$, and then Fubini allows us to conclude that $|E|$ is large, which is our goal.

Unfortunately, when translated into coordinates, the maps $\Pi_{\omega_0}$ are nonlinear. The idea is then to apply a projective transformation sending $y$ to the point at infinity, so that lines through $y$ become vertical lines. After this transformation, the maps $\Pi_{\omega_0}$ become linear projections (in an appropriate coordinate system for $A(2,1)$), and we can then apply Bourgain's projection theorem. The  projective transformation introduces some distortion, but this can be controlled by dealing only with $R_\om$ such that $\om$ stays ``far'' from $y$. We can then conclude that $E$ must be large enough for the conclusion of Theorem \ref{thm:main} to hold.

\begin{remark}
Very recently, the second author \cite{Shmerkin20} developed a \emph{non-linear} version of Bourgain's projection theorem, see in particular \cite[Theorem 1.7]{Shmerkin20}. Using this theorem, it should be possible to avoid the projective transformation and deal directly with the original family of nonlinear projections $\{\Pi_{\omega_0}\}$. While this would make the proof somewhat shorter, we opted for a more self-contained proof based on Bourgain's original formulation.
\end{remark}

\subsection{Setup}

We set the scene for the proof of Theorem \ref{thm:main}. Fix $\alpha\in (0,1/2]$ and let $\gamma$ be small in terms of $\alpha$. Fix $\e$ small and suppose $\de$ is small in terms of all previous parameters. Let $E$ be a discretized $(\delta,\alpha,2\alpha)_2$-Furstenberg set, with associated sets $\Omega$ and $(R_\om)_{\om\in\Omega}$.  According to Lemmas \ref{lem:discrete-to-continuous} and \ref{lem:easy-bound}, our task is to show that if
\be \label{eq:initial-assumption}
|E\times E| = \delta^{-2\gamma} \left|\bigcup_{\om\in\Om} (R_\om\times R_\om)\right|,
\ee
then $\gamma$ cannot be too small, i.e. $\gamma \ge \gamma_0(\alpha)$.

In the course of the proof, we will work with a parameter $\e$, which is an arbitrarily small number whose role is to ensure that $\delta^\e X\le 1$ whenever $X\lesssim 1$. Thus, any expression of the form $O(\e)$ can be considered as negligible. We will also encounter various parameters $\eta_i$. These numbers depend continuously on $\alpha,\gamma$ and $\e$, and will always have the property that (for fixed $\alpha$) they tend to $0$ as $\gamma,\e\to 0$, so they can be made arbitrarily small. In fact, $\eta_i$ will always be controlled by $C_\alpha(\gamma+\e)$ for some $C_\alpha$ depending continuously on $\alpha$. Moreover, $C_\alpha$ will always be linear in $1/\alpha$.

\subsection{Initial processing of the set $E$}

We perform an initial reduction. By splitting $[0,\pi)$ into  $\pi/4$ arcs and considering the arc with the largest number of $\om\in\Om$ with direction in that arc, we may assume that all the directions lie in that $\pi/4$ arc to begin with. Since rotating the picture does not change anything, we henceforth assume that all the $\om\in\Om$  make an angle $\le\pi/4$ with the $y$-axis.

To begin, we observe that $|R_{\omega}|\approx \delta^{2-\alpha}$ (since it is a $(\delta,\alpha)_2$-set) and $|\Omega|\approx \delta^{-2\alpha}$ (since $\Omega^{(\delta)}$ is a $(\delta,2\alpha)$-set and $|\Omega^{(\delta)}|\approx|\Omega|\delta^2$). Combined with Lemma \ref{lem:easy-bound}, it follows that if \eqref{eq:initial-assumption} holds, then
\be \label{eq:size-E}
|E| \approx \delta^{2-2\alpha-\gamma}.
\ee

We define a relation among elements of $E$ by
\[
x\sim y \Leftrightarrow \exists \omega \in \Omega \text{ such that } x,y \in R_{\omega}.
\]
We also define the set of points of $E$ that are related with a lot of points of $E$:
\be \label{eq:def-E1}
E_1:=\{x_0 \in E: \ \delta^{2-2\alpha+\gamma+\e} \leq |\{x_1 \in E : \ x_0 \sim x_1\}| \}.
\ee
\begin{lemma}\label{lem:measureEE'}
$|E_1|\ge \tfrac{1}{2}  \delta^{2\gamma}|E|\approx\delta^{2-2\alpha+\gamma}$.
\end{lemma}
\begin{proof}
By \eqref{eq:initial-assumption},
\[
|E \times E| = \delta^{-2\gamma}|\{(x_0,x_1)\in E\times E : x_0 \sim x_1\}|.
\]
Then, assuming $\delta$ is small enough that $|E|\ge \delta^{2-2\alpha-\gamma+\e/2}$,
\begin{align*}
\delta^{2 \gamma}|E|^2 &= |\{(x_0, x_1) \in E \times E: \ x_0 \sim x_1 \}|\\
&\le |E_1| |E|+\delta^{2-2\alpha+\gamma+\e} |E|\\
&\le |E_1| |E|+\delta^{2 \gamma +\e/2} |E|^2
\end{align*}
This gives the claim if $\de$ is small enough that $\de^{\e/2}\le 1/2$.
\end{proof}

For each $x \in E$ we define the set
\be \label{eq:def-om-x}
\Omega_x:= \{ \omega \in \Omega:  x \in R_{\omega} \},
\ee
which is $\delta$-separated, since $\Omega$ is.

\begin{lemma}\label{lem:omega-bound}
$| \Omega_x| \lesssim \delta^{-\alpha-\gamma}$ for all $x \in E$.
\end{lemma}

\begin{proof}
We know that $|R_\omega| \approx \delta^{2-\alpha}$ and $|R_{\omega} \cap B_r| \lesssim \delta^{2-\alpha} r^\alpha$ for all $r \in [\delta, 2]$. If we take $r=\log(\frac{1}{\delta})^{-C}$ for a sufficiently large constant $C$, we have $1 \lesssim r \leq 1$, and $|R_{\omega} \cap B_r| \leq \frac{|R_{\omega}|}{2}$, so $\delta^{2-\alpha} \approx \frac{|R_\omega|}{2} \leq |R_\omega \setminus B_r|$.

By elementary geometry, there is an absolute constant $C$ such that
\[
\angle(\omega, \omega')\geq C\delta/r \Longrightarrow \diam(R_{\omega} \cap R_{\omega'}) \leq r.
\]
Pick a maximal $(10C\delta/r$)-separated subset $\Omega'_x$ of $\Omega_x$. Then $|\Omega'_x | \gtrsim | \Omega_x|$ (since $r\approx 1$). Also, if $\om\neq\om'\in\Omega'_x$, then $\angle(\om,\om')\ge C\delta/r$; otherwise, using that both $\om$ and $\om'$ intersect $B(0,2)$, we would get $d(\om,\om')<10C\delta/r$. We have seen that $\{ R_\om \setminus B(x,r)\}_{\om\in\Omega'_x}$ is pairwise disjoint, and conclude
\begin{align*}
\delta^{2-2\alpha-\gamma}&\approx |E|\geq |\bigcup_{\omega \in \Omega_x} R_{\omega}|\\
&\geq \sum_{\omega \in \Omega'_x}|R_{\omega}\setminus B(x,r)| \gtrsim |\Omega_x| \delta^{2-\alpha}.
\end{align*}
This yields the claim.
\end{proof}

\begin{lemma}\label{lem:non-concentration-E}
Fix $x_0\in E$. Then, for all $r\in [\delta,2]$,
\[
|\{ x\in E: x\sim x_0\}\cap B_r | \lesssim r^\alpha |E|.
\]
\end{lemma}
\begin{proof}
For each $\om\in\Omega_{x_0}$, we know from non-concentration for $R_\om$ that
\[
|R_\om \cap B_r|\lesssim \delta^{2-\alpha}r^\alpha.
\]
Hence, using Lemma \ref{lem:omega-bound},
\[
|\{ x\in E: x\sim x_0\} \cap B_r| \lesssim |\Omega_{x_0}| \delta^{2-\alpha}r^\alpha\le \delta^{2-2\alpha-\gamma} r^\alpha\approx r^\alpha|E|.
\]
\end{proof}

\begin{lemma}\label{lem:InfBoundA}
If we define
\[
A:=\{(x_0,x_1,x_2)\in E_1\times E^2: \ x_0\sim x_1,\ x_0\sim x_2 \},
\]
then $|A|\gtrsim  \delta^{4-4\alpha+2\gamma+2\e}|E_1|\gtrsim \delta^{6-6\alpha+3\gamma+2\e}$.
\end{lemma}
\begin{proof}
Using Fubini's Theorem, \eqref{eq:def-E1} and Lemma \ref{lem:measureEE'}, we get
\[
|A|=\int_{E_1}|\{x \in E: \ x_0 \sim x\}|^2 \ dx_0 \geq \delta^{4-4\alpha+2\gamma+2\e}|E_1|\gtrsim \delta^{6-6\alpha+3\gamma+2\e}.
\]
\end{proof}

In the next lemma we show that $E_1$ cannot be concentrated in a small strip. This is important because it rules out potential counterexamples of ``train track'' type, see \cite[Figure 1]{KatzTao01}.
\begin{lemma}\label{lem:strip}
Let $L\in A(2,1)$ and consider the strip $S=L^{(\delta^{\eta_1})}$, $\eta_1>0$.  Then
\[
|E_1\cap S|\lesssim \delta^{2-2\alpha-\gamma-\e+\frac{\eta_1}{2}\alpha}.
\]
If $\eta_1\ge \frac{2}{\alpha}(2\e+2\gamma)$, then $|E_1\setminus S|\ge |E_1|/2\gtrsim  \delta^{2-2\alpha+\gamma}$.
\end{lemma}

\begin{proof}[Proof of the lemma]
We know that $|R_\omega| \approx \delta^{2-\alpha}$. Recall that $\Omega^{(\delta)}$ is a $(\delta,2\alpha)$-set in the space $A(2,1)$. In particular,
$|\Omega ^{(\delta)} |\approx \delta^{2-2\alpha}$. Note that the set
\[
\Lambda = \{\omega \in A(2,1):  \angle (\omega, L)\leq \delta^{\frac{\eta_1}{2}}, \ S\cap \omega^{(2\delta)} \neq \varnothing \}
\]
is contained in a ball (in $A(2,1)$) of radius $\approx \delta^{\frac{\eta_1}{2}}$. We deduce from the non-concentration hypothesis that
\[
|\Om^{(\delta)}\cap\Lambda| \lesssim \delta^{2-2\alpha+\frac{\eta_1\alpha}{2}}.
\]
It follows that
\[
|\{\omega \in \Omega: \ \angle (\omega, L)\leq \delta^{\frac{\eta_1}{2}}, \ S\cap R_\omega \neq \varnothing \}| \lesssim \delta^{-2\alpha+\frac{\eta_1\alpha}{2}}.
\]
When $\angle (\omega, L)\geq \delta^{\frac{\eta_1}{2}}$, since $S$ is the $\delta^{\eta_1}$-neighborhood of the line $L$, the intersection $S \cap R_{\omega}$ is contained in a ball of radius $\approx \delta^{\eta_1/2}$ and hence, applying non-concentration of $R_\omega$,
\[
|S\cap R_\omega| \lesssim  \delta^{2-\alpha + \frac{\eta_1 \alpha}{2}}.
\]
Furthermore, $|\{\omega \in \Omega, \ \angle (\omega, L)> \delta^{\frac{\eta_1}{2}}\} | \leq | \Omega| \approx \delta^{-2\alpha}$.

Putting together these facts and the definition of $E_1$ from \eqref{eq:def-E1}, we estimate
\begin{align*}
|E_1\cap S| \delta^{2-2\alpha+\gamma+\e}&\leq \int_{E_1 \cap S} |\{x \in E: \ x_0 \sim x\}| \ dx_0\\
&= |\{(x_0,x_2)\in E_1\times E: \ x_0 \in S, \ x_0\sim x_2\}|\\
&\leq \sum_{\omega \in \Omega}|S\cap R_\omega| |R_\omega|\\
&\approx \delta^{2-\alpha} \left(\sum_{\omega \in \Omega, \ \angle (\omega, L)> \delta^{\frac{\eta_1}{2}}}|S\cap R_\omega| + \sum_{\omega \in \Omega, \ \angle (\omega, L)\leq \delta^{\frac{\eta_1}{2}}}|S\cap R_\omega| \right)\\
&\lesssim \delta^{2-\alpha} ( \delta^{-2\alpha} \delta^{2-\alpha + \frac{\eta_1 \alpha}{2}} + \delta^{-2\alpha+ \frac{\eta_1 \alpha}{2}} \delta^{2-\alpha}) \approx\delta^{4-4\alpha+ \frac{\eta_1  \alpha}{2}}.
\end{align*}
We conclude that $|E_1 \cap S| \lesssim \delta^{2-2\alpha-\gamma-\e+ \frac{\eta_1  \alpha}{2}}$.

Under the assumption $\eta_1\ge\frac{2}{\alpha}(2\e+2\gamma)$, the upper bound for $|E_1\cap S|$ is much smaller than $|E_1|$, so the second claim follows.
\end{proof}

We denote the line trough $x_1$ and $x_1'$ by $L_{x_1, x_1'}$.
\begin{lemma} \label{lem:selection-y1-y2}
If $\eta_1 \geq \alpha^{-1}(12\gamma+8\e)$ and $\eta_2 \geq \alpha^{-1}(4\gamma+3\e)$, then there exist $y_1, y_2 \in E$ such that $|y_1-y_2|\geq\delta^{\eta_2}$ and
\[
|\{x_0 \in E_1: x_0 \sim y_1, x_0\sim y_2, \ x_0 \notin L_{y_1, y_2}^{(\delta^{\eta_1})} \}| \geq \delta^{2-2\alpha+5\gamma+2\e}.
\]
\end{lemma}
\begin{proof}
We define the sets
\begin{align*}
A&=\{(x_0,x_1,x_2) \in E_1 \times E \times E: x_0\sim x_1, x_0\sim x_2\},\\
B&=\{(x_0,x_1,x_2) \in E_1 \times E \times E :   x_0 \in L_{x_1, x_2}^{(\delta^{\eta_1})} \},\\
C&=\{(x_0,x_1,x_2) \in A :   |x_1-x_2|<\delta^{\eta_2}\}.
\end{align*}
Recall from Lemma \ref{lem:InfBoundA} that $|A|\gtrsim \delta^{6-6\alpha+3\gamma+2\e}$. By Lemma \ref{lem:strip}, Fubini, and the assumption on $\eta_1$,
\[
|B| \lesssim \delta^{2-2\alpha-\gamma-\e+\frac{\eta_1}{2}\alpha}|E|^2\lesssim \delta^{6-6\alpha-3\gamma-\e+\frac{\eta_1}{2}\alpha} \lesssim \delta^\e |A|,
\]
and hence $|B|\le |A|/3$ if $\delta$ is small. Now from Lemmas \ref{lem:non-concentration-E} and \ref{lem:InfBoundA}, Fubini,  and the assumption on $\eta_2$, we get that
\[
|C| \lesssim |E_1||E|^2 \delta^{\eta_2 \alpha} \lesssim \de^\e |A|,
\]
so $|C|\le |A|/3$ if $\de$ is small. We have seen that
\[
|A\setminus (B\cup C)| \ge |A|/3 \gtrsim \delta^{6-6\alpha+3\gamma+2\e}.
\]
We conclude from Fubini and \eqref{eq:size-E} that there is $(y_1,y_2)\in E^2$ such that
\[
|\{ x_0\in E_1: (x_0,y_1,y_2)\in A\setminus (B\cup C)| \ge \frac{|A\setminus (B\cup C)|}{|E\times E|} \gtrsim \delta^{2-2\alpha+5\ga+2\e}.
\]
\end{proof}

We fix the points $y_1$, $y_2$ given by the previous lemma for the rest of the proof, and define the set
\be  \label{def:E2}
E_2:= \{x_0 \in E_1:  x_0 \sim y_1, x_0\sim y_2, \ x_0 \notin L_{y_1, y_2}^{(\delta^{\eta_1})}  \}.
\ee
Then, by Lemma \ref{lem:selection-y1-y2},
\be \label{eq:measure-E2}
|E_2| \gtrsim \delta^{2-2\alpha+5\gamma+2\e}.
\ee
Thus if $\gamma$ is small, then $E_2$ is quite dense in $E$, which says that a large part of $E$ is related to the fixed pair of well-separated points $y_1,y_2$. 
\begin{lemma}\label{lem:InfBoundSumOmega1}
\[
\sum_{\omega \in \Omega}|R_{\omega} \cap E_2| \gtrsim \delta^{\gamma-\alpha+\e}|E_2| \gtrsim \delta^{2-3\alpha+6\gamma+3\e}.
\]
\end{lemma}
\begin{proof}
Recall the definition of $\Omega_{x_0}$ from \eqref{eq:def-om-x}. For every $x_0 \in E_2$, since $E_2 \subseteq E_1$ we have that
\[
\delta^{2-2\alpha+\gamma+\e}\leq |\{x\in E: x_0 \sim x\}| \leq \sum_{x \in \Omega_{x_0}}|R_{\omega}| \approx \delta^{2-\alpha}|\Omega_{x_0}|,
\]
so $|\Omega_{x_0}|\gtrsim \delta^{\gamma-\alpha+\e}$. 
We conclude
\[
\sum_{\omega \in \Omega}|R_{\omega} \cap E_2|= \int_{E_2}\sum_{\omega \in \Omega}\1_{R_{\omega}}(x_0) \ dx_0\nonumber \gtrsim \delta^{\gamma-\alpha+\e}|E_2|.
\]
\end{proof}

Let $\eta_3$ a small number to be defined later.
We want to choose a large set $\Omega_1\subseteq \Omega$ such that:
\begin{itemize}
\item $\sum_{\omega \in \Omega_1}|R_{\omega} \cap E_2|$ is large
\item if $\omega \in \Omega_1$, then $\omega$ does not intersect at least one of the balls $B(y_i,\delta^{\eta_3})$.
\end{itemize}

\begin{lemma}
If $\eta_3 \geq \eta_1 +\eta_2+\e$, and $\omega \cap B(y_i, \delta^{\eta_3}) \neq \varnothing$ for $i=1,2$, then $R_\omega \cap E_2 = \varnothing$.
\end{lemma}

\begin{proof}
Let $L$ be the line joining $y_1$ and $y_2$. By hypothesis and elementary geometry one can see that if $\delta$ is small enough,
\[
\om \subset L^{(4\delta^{\eta_3 - \eta_2})} \subset L^{(\delta^{\eta_3 - \eta_2-\e})}\subset L^{(\delta^{\eta_1})}.
\]
And by construction, we have $E_2 \cap L^{(\delta^{\eta_1})}=\varnothing$. So, $\omega \cap E_2=\varnothing$.
\end{proof}

By the previous Lemma, we can split $\Omega$ as a disjoint union
\[
\Omega = \Omega_{B_{y_1}} \cup \Omega_{B_{y_2}} \cup \Omega',
\]
where
\[
\Omega_{B_{y_i}}:=\{\omega \in \Omega: \om \cap B(y_i, \delta^{\eta_3}) \neq \varnothing\},
\]
and
\[
\Omega':=\{\omega \in \Omega: \om \cap \left( B(y_1, \delta^{\eta_3})\cup B(y_2, \delta^{\eta_3})\right) = \varnothing\}.
\]
Note that for each of these sets, all the lines in it miss at least one of the two balls $B(y_i,\delta^{\eta_3})$. Hence, recalling Lemma \ref{lem:InfBoundSumOmega1} and pigeonholing, we deduce:
\begin{corollary} \label{cor:Omegasavoidingtwoballs}
If $\eta_3 \geq \eta_1 +\eta_2+\e$, there are $i\in\{1,2\}$ and $\Omega_1\subset\Om$ such that
\[
\om \cap B(y_i,\delta^{\eta_3}) =\varnothing\quad\text{for all }\om\in\Om_1,
\]
\[
\sum_{\omega \in \Omega_1} |R_{\omega} \cap E_2|\gtrsim \delta^{\gamma-\alpha+\e}|E_2|\gtrsim \delta^{2-3\alpha+6\gamma+3\e}.
\]
\end{corollary}

Fix the set $\Omega_1$ provided by Corollary \ref{cor:Omegasavoidingtwoballs}. We can assume without loss of generality that
for all $\om \in \Omega_1$, $\om \cap B(y_1, \delta^{\eta_3}) = \varnothing$, and denote $y=y_1$ from now on.
\begin{lemma} \label{lem:remove-problem-line}
Let $\ell_0$ be the horizontal line through $y$. Assuming that $\eta_4\ge \alpha^{-1}(6\gamma+5\e)$, there is a set $E_3\subset E_2$ such that $E_3\cap\ell_0^{(\delta^{\eta_4})}=\varnothing$,
\[
|E_3|\ge |E_2|/2 \gtrsim \delta^{2-2\alpha+5\gamma+2\e},
\]
and
\[
\sum_{\omega \in \Omega_1} |R_{\omega} \cap E_3|\gtrsim \delta^{\gamma-\alpha+\e}|E_3|\gtrsim  \delta^{2-3\alpha+6\gamma+3\e}.
\]
\end{lemma}
\begin{proof}
Let $E_3=E_2\setminus \ell_0^{(\delta^{\eta_4})}$. The fact that $|E_3|\ge |E_2|/2$ is immediate from Lemma \ref{lem:strip} and the assumption on $\eta_4$.

Recall that we are assuming that all the $\om\in\Om$ make an angle $\le \pi/4$ with the vertical direction. This implies that the angle between $\om\in\Omega_1$ and the line $\ell_0$ is bounded below by $\pi/4$. By the non-concentration assumption for $R_\om$, we deduce that
\[
|R_\om \cap \ell^{(\delta^{\eta_4})}| \lesssim \delta^{2-\alpha+\eta_4 \alpha},
\]
and hence
\[
\sum_{\omega\in\Omega_1} |R_\om \cap \ell^{(\delta^{\eta_4})}| \lesssim \delta^{2-3\alpha+\eta_4\alpha}.
\]
It follows from \eqref{eq:measure-E2}, Corollary \ref{cor:Omegasavoidingtwoballs} and the choice of $\eta_4$ that
\[
\sum_{\omega\in\Omega_1} |R_\om \cap \ell^{(\delta^{\eta_4})}| \lesssim \delta^\e \delta^{\gamma-\alpha+\e}|E_2|,
\]
and this yields the claim.
\end{proof}

For simplicity of notation, we translate the coordinate system so that $y$ becomes the origin $0$, and hence the line $\ell_0$ from Lemma \ref{lem:remove-problem-line} becomes the $x$-axis. This does not change any of our previous estimates, other than the fact that now $E_3$ is no longer contained in $B(0,2)$, but (together with Lemma \ref{lem:remove-problem-line}) we still have
\be \label{eq:newlargeball}
E_3\su B(0,4)\cap \{(p_1,p_2) \in \rr^2: |p_2| \geq \de^{\eta_4}\}.
\ee
We also recall that (in the new coordinates) each $\om\in\Omega_1$ is at distance at least $\de^{\eta_3}$ from $0$.

We perform a further dyadic pigeonholing to localize both $E_3$ and $\Omega_1$.
\begin{lemma} \label{lem:regular-pigeonholing}
There exist $y_0\in [\de^{\eta_4},2]$ and $b_0\in [\delta^{\eta_3}, 2]$ such that, if we define
\begin{align*}
E' &= \{ (x,y) \in E_3: y\in [y_0,2 y_0]\},\\
\Omega' &= \{ \omega=\{ x=ay+b\} \in\Omega_1: b\in [b_0,2b_0]\},
\end{align*}
then
\[
\sum_{\omega \in \Omega'} |R_{\om} \cap E'|\gtrsim  \delta^{2-3\alpha+6\gamma+3\e}.
\]
\end{lemma}
\begin{proof}
If $\om: x=ay+b$, we write $a=a(\om)$, $b=b(\om)$. Fix $\om\in\Omega_1$. Note that $|a(\om)|\le 1$ since $\om$ makes an angle $\le \pi/4$ with the $y$-axis. Hence $|b(\om)|\le 8$, for otherwise the line $\om$ cannot intersect the ball $B(0,4)$, since $|a(\om)y+b(\om)|\ge |b(\om)|-|y|>4$ for $|y|\le 4$. On the other hand, since $\om$ does not enter the ball $B(0,\delta^{\eta_3})$, in particular $(b(\om),0)\notin B(0,\delta^{\eta_3})$. In summary, for $\om\in \Omega_1$ we have $|b(\om)|\in [\delta^{\eta_3},8]$. Hence, if we split $\Omega_1$ as
\[
\Omega_1 = \bigcup_j \Om_{2,j,*}:= \bigcup_j \{\om\in\Omega_1: b(\om) \in *[2^j,2^{j+1}]\},
\]
where $*$ is either $+$ or $-$, there are $\le 3\log(1/\delta)$ values of $j$ for which $\Om_{2,j,*}$ is nonempty and hence, applying Lemma \ref{lem:remove-problem-line}, we can fix $\Om'=\Om_{2,j,*}$ such that
\[
\sum_{\omega \in \Omega'} |R_{\om} \cap E_3|\gtrsim  \delta^{2-3\alpha+6\gamma+3\e}.
\]
Now we perform the same argument for the points $(x_1,x_2)\in E_3$; we know that $\delta^{\eta_4}\le |x_2|\le 4$ so pigeonholing as before we get the set $E'\subset E_3$ as claimed.
\end{proof}
From now on we work with the sets $E'$ and $\Om'$ and the parameters $y_0,b_0$ provided by the lemma.

\subsection{Projective transformation and application of Bourgain's projection theorem}

Now we will apply a projective transformation sending lines through the origin to vertical lines and preserving horizontal lines. To make the argument more concrete, we work with the following real plane map. Recall that $\ell_0$ denotes the $x$-axis. Let
\begin{equation} \label{eq:def-psi}
\psi: \R^2\setminus\ell_0 \to \R^2, (x,y) \mapsto \left(\frac{x}{y},\frac{1}{y}\right).
\end{equation}
\begin{lemma}
\label{lem:psi-projective}
The map $\psi$ sends lines to lines. More precisely, if the non-horizontal line $\om$ is given by $\{(x,y):x=ay+b\}$, then $\psi(\om\setminus\ell_0)$ is given by $\{(x,y):x=by+a\}\setminus\ell_0$.

In particular, if $\om\in\Om'$, then the modulus of the slope of $\psi(\om)$ lies in $[b_0^{-1}/2,b_0^{-1}]$.
\end{lemma}
\begin{proof}
This is a direct calculation. For the last claim, recall Lemma \ref{lem:regular-pigeonholing}.
\end{proof}

We will also need to know that the transformation $\psi$ does not cause too much distortion on the set $E'$.
\begin{lemma} \label{lem:psi-distortion}
For all $p,q\in B^2(0,4)$ with vertical coordinate in $[y_0,2 y_0]$ (in particular for $p,q\in E'$) it holds that
\begin{enumerate}[(a)]
\item $y_0^{-2}|p-q| \le |\psi(p)-\psi(q)| \leq 36 y_0^{-2} |p-q|$.
\item $|\det(\psi'(p))|\in [y_0^{-3}/8,y_0^{-3}]$.
\end{enumerate}
\end{lemma}
\begin{proof}
These are straightforward calculations.
\end{proof}

\begin{corollary} \label{cor:rectangles-to-rectangles}
If $R$ is a $1\times 4\delta$ rectangle, then $\psi(R\cap E')$ is contained in a rectangle of size $C y_0^{-2}\times C y_0^{-2}\delta$, where $C>0$ is absolute. If the central line of $R$ passes through the origin (when extending it beyond $R$), then $\psi(R\cap E')$ is contained in a vertical strip of width $C y_0^{-2}\delta$.
\end{corollary}
\begin{proof}
By making the rectangle smaller (which only helps our task) we may assume that $y\in [y_0,2y_0]$ whenever $(x,y)\in R$. By Lemma \ref{lem:psi-distortion}, the long central segment of the rectangle is mapped to a segment of length $\lesssim y_0^{-2}$, and the segments of length $\delta$ between the extremes of the central segment and the corners of the rectangle are mapped to segments of length $\lesssim y_0^{-2}\delta$. With a little of planar geometry, this gives the first claim. The second claim follows from the first and Lemma \ref{lem:psi-projective}.
\end{proof}

The following lemma shows why we wanted to send lines through the origin to vertical lines: it says that $\psi(E')$ has a product structure.
\begin{lemma} \label{lem:product-structure}
Let $\delta_1=y_0^{-2}\delta$. There exists a $\delta_1$-discretized set $A\subset [-2+\delta,2-\delta]$ such that $\psi(E')\subset A\times \R$, and
\[
|A|\lesssim \delta_1\cdot \delta^{-\alpha-\gamma}\le \delta^{1-\alpha-\gamma-2\eta_4}.
\]
\end{lemma}
\begin{proof}
Since every $x\in E'\subset E_2$ satisfies $x\sim 0$, it follows from Lemma \ref{lem:omega-bound} that $E'$ is contained in the union of $\lesssim \delta^{-\alpha-\gamma}$ of the sets $R_\omega, \om\in\Om$ containing $0$. Note that for each such $\om$ there is a $1\times 4\delta$ rectangle $\wt{R}_\om$, whose central line contains $0$ and is parallel to $\om$, such that
\[
R_\omega\cap E' \subset \wt{R}_\om \cap \{ p\in\R^2:|p|\ge \delta^{\eta_4}\}.
\]
Hence we can apply Corollary \ref{cor:rectangles-to-rectangles} to cover $\psi(E')$ with $\lesssim \delta^{-\alpha-\gamma}$ vertical strips of width $\delta_1$. Furthermore, if $\om$ has slope $1/a$, then so does the central line of $\wt{R}_\om$, and because the latter goes through the origin, it gets mapped under $\psi$ to the line $\{ (u,v): u=a\}$ by Lemma \ref{lem:psi-projective}. But $|a|\le 1$ by our standing assumption that all $\om$ make an angle $\le \pi/4$ with the $y$-axis, and thus $I$ intersects the interval $[-1,1]$. This concludes the proof.
\end{proof}

The set $A$ will eventually provide the measure $\mu$ on $S^1$ to which we will apply Theorem \ref{thm:bourgain}. However, a priori $A$ does not need to satisfy any decay conditions, so our next aim is to apply Lemma \ref{lem:refinement} to replace it by a subset $A^*$ that does. The next lemma is a first step towards this.
\begin{lemma} \label{lem:concentration-pullback-strip}
Let $I\subset [-2,2]$ be an interval of length $\delta'\in[\delta,2]$. Then, for any $\om\in\Om'$,
\[
| R_\om \cap E'\cap \psi^{-1}(I\times\R) | \lesssim  \delta^{2-\alpha} (\delta^{-\eta_3}\delta')^\alpha.
\]
\end{lemma}
\begin{proof}
It follows from Lemmas \ref{lem:psi-projective} and \ref{lem:psi-distortion} that if $I$ has midpoint $x_0$, then $\ell_I = \psi^{-1}\{ x_0\times \R\}$ is a line going through the origin, and
\[
\psi^{-1}(I\times\R)\cap E' \subset \ell_I^{(C\delta')} \cap B(0,4).
\]
If the strip $\ell_I^{(C\delta')}$ does not meet $R_\om$, there is nothing to do. Otherwise, since $\om$ is disjoint from $B(0,\delta^{\eta_3})$, the angle between $\ell_I$ and $\om$ is  $\gtrsim \delta^{\eta_3}$, and  hence $R_\om\cap  \ell_I^{(C\delta')}$ has diameter $\lesssim \delta^{-\eta_3}\delta'$. The claim now follows from the non-concentration property of $R_\om$.
\end{proof}

\begin{lemma} \label{lem:sum-Omega4}
Suppose $\eta_5 \ge 7\ga+\alpha\eta_3+2\eta_4+4\e$. Then there exists a $(\delta,\alpha+\gamma+2\eta_4,\eta_5)$-set $A^*\subset A^{(\delta)}$ such that
\[
\sum_{\om\in \Om'} |R_\om\cap E'\cap \psi^{-1}(A^{*} \times \R)| \gtrsim \delta^{2-3\alpha+6\gamma+3\e}.
\]
\end{lemma}
\begin{proof}
Let $A^*, A^{**}$ be the sets provided by Lemma \ref{lem:refinement}, applied with $n=1$, $s=\alpha+\gamma+2\eta_4$, and $\eta_5$ in place of $\eta$. Note that the lemma is indeed applicable by Lemma \ref{lem:product-structure}. Hence $A^*$ is a $(\delta,\alpha+\gamma+2\eta_4,\eta_5)$-set contained in $A^{(\delta)}$, and we only have to verify the last claim. By Lemma \ref{lem:refinement}, the set $A^{**}$ is a union of $A^{**}_{\delta'}$ where $\delta'$ ranges over dyadic numbers in $[2\delta,2]$ and each $A^{**}_{\delta'}$ can be covered by
\[
\delta^{\eta_5} (\delta')^{-\alpha-\gamma-2\eta_4}
\]
intervals of length $2\delta'$. We deduce from Lemma \ref{lem:concentration-pullback-strip} that for every $\omega \in \Om'$
\[
|R_\om\cap E'\cap  \psi^{-1}(A^{**}_{\delta'} \times \R) | \lesssim \delta^{2-\alpha+\eta_5-\eta_3\alpha} (\delta')^{-\gamma-2\eta_4}
\]
and hence, adding up over all dyadic $\delta'\in [2\delta,2]$ and all $\om\in\Om'$,
\begin{align*}
\sum_{\om\in\Om'} |R_\om\cap  E'\cap \psi^{-1}(A^{**} \times \R) | &\lesssim |\Om'| \delta^{2-\alpha+\eta_5-\eta_3\alpha-\gamma-2\eta_4}\\
&\le  \delta^\e \delta^{2-3\alpha+6\gamma+3\e},
\end{align*}
using that $|\Om'|\lesssim \delta^{-2\alpha}$ and the assumption on $\eta_5$ in the last line.

Recall from Lemma \ref{lem:regular-pigeonholing} that
\begin{equation} \label{eq:lower-bound-sum}
\sum_{\omega \in \Om'} |R_{\omega} \cap E'|\gtrsim \delta^{2-3\alpha+6\gamma+3\e}.
\end{equation}
Since $A\subset A^*\cup A^{**}$ and $E'\subset \psi^{-1}(A\times \R)$ by Lemma \ref{lem:product-structure}, the proof is complete.
\end{proof}

Now we refine $A^*$ further, with the goal of ensuring that each pullback $\psi^{-1}(I\times \R)$ meets a uniformly large number of $R_\om$ for each $\delta$-interval $I$ in this refinement.
\begin{prop} \label{prop:uniformize-A}
There is a collection $\mathcal{J}$ of disjoint $\delta$-intervals such that if $I\in\mathcal{J}$ then $I\subset A^*$ and
\[
| \{ \om\in\Om':  R_\omega \cap E'\cap  \psi^{-1}(I\times \R) \neq\varnothing \} | \ge   \delta^{-2\alpha+7\gamma+\alpha\eta_3+2\eta_4+3\e}
\]
and, moreover,
\[
|\mathcal{J}| \gtrsim \delta^{-\alpha+6\gamma+\alpha\eta_3 +3\e}.
\]
\end{prop}
\begin{proof}
Recall that $A^{*}$ is $\delta$-discretized, so we can write $A^*=\cup_{j=1}^M I_j$, where the $I_j$ are $\delta$-intervals with bounded overlap. Note that $M\approx |A^*|\delta^{-1}$. By Lemma \ref{lem:sum-Omega4}
\[
\sum_{j=1}^M \sum_{\omega \in \Om'} |R_{\omega} \cap E'\cap  \psi^{-1}(I_j \times \R)| \gtrsim  \delta^{2-3\alpha+6\gamma+3\e}.
\]
On the other hand, it follows from Lemma \ref{lem:concentration-pullback-strip} that, for each fixed $j$,
\[
 \sum_{\omega \in \Om'} |R_{\omega}\cap E' \cap \psi^{-1}(I_j \times \R)| \lesssim  |\Om'| \delta^{2-\alpha \eta_3} \lesssim \delta^{2-2\alpha-\alpha \eta_3}.
\]
A little algebra then shows that there is a set $J\subset \{1,\ldots,M\}$ with
\[
|J| \gtrsim \frac{\delta^{2-3\alpha+6\gamma+3\e}}{\delta^{2-2\alpha-\alpha \eta_3}} = \delta^{-\alpha+6\gamma+\alpha \eta_3+3\e}
\]
such that, for any $j\in J$,
\[
\sum_{\omega \in \Om'} |R_{\omega} \cap E'\cap \psi^{-1}(I_j \times \R)| \gtrsim \frac{\delta^{2-3\alpha+6\gamma+3\e}}{M} .
\]
We take $\mathcal{J}:=\{I_j: \ j\in J\}$.  By passing to a subset of comparable cardinality, we may assume that the $I_j$ are disjoint. We know from Lemma \ref{lem:product-structure} that
\[
|A^*| \le |A^{(\delta)}| \lesssim \delta^{1-\alpha-\gamma-2\eta_4},
\]
whence $M\lesssim \delta^{-\alpha-\gamma-2\eta_4}$, and we deduce that if $j\in J$, then
\[
\sum_{\omega \in \Om'} |R_{\omega} \cap E'\cap \psi^{-1}(I_j \times \R)| \gtrsim \delta^{2-2\alpha+7\gamma+2\eta_4+3\e}.
\]
On the other hand, Applying Lemma \ref{lem:concentration-pullback-strip} with $\de'=\de$, we see that, for each $j$,
\[
\sum_{\omega \in \Om'} |R_{\omega} \cap E'\cap \psi^{-1}(I_j \times \R)| \lesssim  \delta^{2-\eta_3\alpha}|\{\om\in\Om': R_\om \cap E'\cap \psi^{-1}(I_j\cap \R)\neq\varnothing\}|.
\]
Combining the last two displayed equations, we reach the desired conclusion.
\end{proof}

We have now constructed the measure $\mu$ that will feature in the application of Bourgain's projection theorem:
\begin{corollary} \label{cor:decay-measure}
Let $\mathcal{J}$ be the collection given by Proposition \ref{prop:uniformize-A}, write $\wt{A}$ for the union of the intervals in $\mathcal{J}$, and let $\wt{\mu}$ be the normalized restriction of Lebesgue measure to $\wt{A}$ (that is, $\wt{\mu}=\tfrac{1}{|\wt{A}|}\1|_{\wt{A}}dx$). Then
\[
\wt{\mu}(B(x,r)) \lesssim \delta^\e\delta^{-\lambda_1} r^\alpha\quad\text{for all }x\in [-2,2], r\in [\delta,1],
\]
where
\[
\lambda_1 = 7\gamma+\alpha\eta_3+2\eta_4+\eta_5+4\e.
\]
If $\mu$ is the measure on $S^1$ given by $\mu(X)=\wt{\mu}(a: \arctan(a)\in X)$, then the same decay estimate holds for $\mu$.
\end{corollary}
\begin{proof}
The second assertion follows from the first and the fact that $\arctan$ is bi-Lipschitz on $[-2,2]$. For the first, it follows from Proposition \ref{prop:uniformize-A} that
\[
|\wt{A}| \gtrsim \delta^{1-\alpha+6\gamma+\alpha\eta_3+3\e},
\]
while from Lemma \ref{lem:sum-Omega4}, non-concentration for $A^*$ and the fact that $\wt{A}\subset A^*$, we get
\[
|\wt{A}\cap B_r|\lesssim \delta^{1-\alpha-\gamma-2\eta_4-\eta_5} r^{\alpha+\gamma+2\eta_4}\le \delta^{1-\alpha-\gamma-2\eta_4-\eta_5} r^{\alpha}.
\]
Combining the last two displayed equations yields the claim.
\end{proof}

The next lemma introduces the set $F$ that will feature in our application of Theorem \ref{thm:bourgain}. It is nothing but a convenient parametrization of $\Omega'$.
\begin{lemma} \label{lem:decay-set}
Given a non-vertical line $\om=\{(x,y): y=ax+b\} \in A(2,1)$, denote $h(\om)=(a,b)$, and set $\varphi=h \circ \psi$. Then the set $F=\varphi(\Omega')\subset\R^2$ satisfies
\begin{align*}
F&\subset B(0,\delta^{-\lambda_2}),\\
\cN_\delta(F)&\ge \delta^{-\lambda_3} \delta^{-2\alpha},\\
\cN_\delta(F\cap B_r) &\le \delta^{-\lambda_3}  r^{2\alpha} \cN_\delta(F),
\end{align*}
where $\lambda_2=2\eta_3+\e$, $\lambda_3=6\gamma+4\e$.
\end{lemma}
\begin{proof}
Note that $\varphi^{-1}=\psi \circ h^{-1}$ by Lemma \ref{lem:psi-projective}. A calculation shows that
\be \label{eq:formula-line}
\varphi^{-1}(v)=\ell_v :=\{ p\in\R^2: p\cdot v=1\}.
\ee

Let $A'\subset A(2,1)$ be the set of lines that hit $B(0,4)$ and avoid $B(0,\delta^{\eta_3})$. We know that $\Om'\subset A'$. Note that all lines in $A'$ are of the form $\ell_v$. If $\ell_v\in A'$, then there is $x_v\in \ell_v\cap B(0,5)$, hence $1=x_v\cdot v\le |x_v||v|\le 5|v|$. Thus $1/|v|\le 5$. On the other hand, if $\ell_v\in A'$ then, since $\ell_v$ is disjoint from $B(0,\tfrac{1}{2}\delta^{\eta_3})$ and $v/|v|^2\in\ell_v$,  we must have $1/|v|\ge \delta^{\eta_3}/2$. Combining these facts with Lemma \ref{lem:distance-lines} we deduce that, for some universal $C>0$,
\be \label{eq:bi-Lipschitz}
C^{-1} \le \frac{|\varphi(\om)-\varphi(\om')|}{d(\om,\om')} \le C\delta^{-2\eta_3}  \quad\text{for all  }\om,\om'\in A'.
\ee
The set $A'$ has bounded diameter in $A(2,1)$. Thus the right-hand side inequality in \eqref{eq:bi-Lipschitz} yields the first claim.

From Lemma \ref{lem:regular-pigeonholing} and the bound $|R_\om|\lesssim \delta^{2-\alpha}$, we get $|\Omega'|\gtrsim \delta^{-2\al-6\ga-3\e}$. By \eqref{eq:bi-Lipschitz} and since $\Omega'$ is $\de$-separated, $F=\varphi(\Omega')$ is $(\delta/C)$-separated. Hence
\be \label{eq:lower-bound-meas-F}
\cN_\delta(F) \gtrsim |F|= |\Omega'|  \gtrsim \delta^{-2\al-6\ga-3\e},
\ee
giving the second claim.

Finally, since $\Omega'\subset\Omega$ and $\Omega^{(\delta)}$ is a $(\delta,2\alpha)$-set, we have
\[
|\Omega'\cap B_r| \lesssim \delta^{-2\alpha} r^{2\alpha}.
\]
Using \eqref{eq:bi-Lipschitz} again and the lower bound \eqref{eq:lower-bound-meas-F}, we conclude
\begin{align*}
\cN_\delta(F\cap B(x,r)) &\le |F\cap B(x,r)| \\
&= \left|\varphi\left(\Omega'\cap \varphi^{-1}(B(x,r))\right)\right| \\
&\le \left| \phi\left(\Omega' \cap B(\varphi^{-1}(x), C r\right)\right|\\
&\lesssim \delta^{-2\alpha} r^{2\alpha}\\
&\lesssim \delta^{-6\ga-3\e} \cN_\delta(F) r^{2\alpha}.
\end{align*}
In light of the choice of $\lambda_3$, this concludes the proof.
\end{proof}

We can now apply Bourgain's projection theorem and conclude the proof of Theorem \ref{thm:main}.
\begin{proof}[Proof of Theorem \ref{thm:main}]
Let $\lambda$ be the number provided by Theorem \ref{thm:bourgain} applied with $\kappa=\alpha$ and $\beta=2\alpha$. In particular, $\lambda$ depends only on $\alpha$.

Let $\wt{\mu},\mu, F$ and $\lambda_i, i=1,2,3$ be as given in Corollary \ref{cor:decay-measure} and Lemma \ref{lem:decay-set}. We further define
\[
\lambda_4 = 7\gamma+\alpha\eta_3+2\eta_4+4\e.
\]
Since $\e$ is arbitrarily small and all the numbers $\eta_i$ can be made small by making $\gamma$ and $\e$ small (in terms of $\alpha$), it follows that there is a number $\gamma_0=\gamma_0(\alpha)>0$ such that $\lambda_i<\lambda/4-\e$ for all $i$ provided that $\gamma\le \gamma_0$. Under this assumption, we can apply Theorem \ref{thm:bourgain}, together with Remark \ref{rem:bourgain}, to deduce that, whenever
\be \label{eq:dense-subset}
\cN_\delta(F') \ge \delta^{\lambda/4}\cN_\delta(F),
\ee
we have
\[
\cN_\delta(P_e F') \ge \delta^{-\alpha-\lambda/4},
\]
for a set of $e\in S^1$ of $\mu$-measure $\ge 1/2$ (in fact larger, but this is enough for us). Let $\Pi_x(a,b)=ax+b$. Recalling the relation between $\mu$, $\wt{\mu}$, $\wt{A}$ from Corollary \ref{cor:decay-measure}, this implies that
\be \label{eq:large-fibers}
\cN_\delta(\Pi_x F') \gtrsim \delta^{-\alpha-\lambda/4}
\ee
for all $x$ in a set of measure $\ge |\wt{A}|/2$. Recall from Proposition \ref{prop:uniformize-A} that $\wt{A}$ is the union of the disjoint $\delta$-intervals $\in\mathcal{J}$. Hence, for at least half of the intervals $I$ in $\mathcal{J}$, there is a point $x_I\in I$ such that \eqref{eq:large-fibers} holds for $x=x_I$. Let $A'$ be the collection of such points $x_I$. Thus, using the bound on $|\mathcal{J}|$ from Proposition \ref{prop:uniformize-A},
\be \label{eq:size-A}
|A'| \ge \frac{1}{2}|\mathcal{J}| \gtrsim \delta^{-\alpha+6\gamma+\alpha \eta_3+3\e}.
\ee
We underline that \eqref{eq:large-fibers} holds for all $x\in A'$ and for every subset $F'\subset F$ with $\cN_\delta(F')\ge \delta^{\lambda/4}\cN_\delta(F)$. In particular, $F'$ may depend on $x\in A'$.

Let us now consider the sets
\[
Q_\om = \psi( (R_\om\cap E')^{(4\delta)}),
\]
\[
Q = \psi((E')^{(4\delta)}) =\bigcup_{\om\in\Om'} Q_\om.
\]
If we can get a lower bound on $\cN_\delta(Q)$, Lemma \ref{lem:psi-distortion} will give us a lower bound on $|(E')^{(4\delta)}|$ and hence on $|E^{(4\delta)}|\approx |E|$. More precisely, part (b) of Lemma \ref{lem:psi-distortion} yields
\[
|E|\gtrsim y_0^{3}|Q| \gtrsim \delta^{3\eta_4}|Q|.
\]
Since $Q$ is $\delta$-discretized (this follows from $y_0^{-2}\ge 1/4$ and \eqref{lem:psi-distortion}), we therefore have
\be \label{eq:size-E-in-terms-of-Q}
|E|\gtrsim \delta^{2+3\eta_4}\cN_\delta(Q).
\ee
Recalling the definitions from Lemma \ref{lem:decay-set}, we note that, for each non-vertical line $\ell$,
\be \label{eq:property-Pi}
\left(x_0,\Pi_{x_0}(h(\ell))\right) \in \ell.
\ee
On the other hand, we know from Proposition \ref{prop:uniformize-A} and our choice of $\lambda_4$ that for each $I\in\mathcal{J}$, the strip $I\times \R$ meets $\psi(R_\om\cap E')$ for at least $\delta^{-2\alpha+\lambda_4}$ values of $\om$. By $y_0^{-1}\ge 1/2$ and Lemma \ref{lem:psi-distortion}, this means that  for each $x_0\in A'$, the vertical line $\{x=x_0\}$ meets $Q_\om$ for at least $\delta^{-2\alpha+\lambda_4}$ values of $\om$. Let $\Omega'(x_0)$ be the set of all such $\om$, and set
\[
F_{x_0}=\{\varphi(\om):\om\in\Omega'(x_0)\},
\]
so that  $|F_{x_0}|\gtrsim \delta^{-2\alpha+\lambda_4}$. It follows from \eqref{eq:bi-Lipschitz} that $F_{x_0}$ is $(\delta/C)$-separated for some constant $C>0$, and hence
\[
\cN_\delta(F_{x_0}) \ge \frac{1}{C}\delta^{-2\alpha+\lambda_4} \ge\frac{1}{C}\delta^{\lambda_4}\cN_\delta(F).
\]
Since $\lambda_4<\lambda/4-\e$, we have shown that \eqref{eq:large-fibers} holds for $F'=F_{x_0}, x_0\in A'$, that is,
\be \label{eq:large-fiber-bis}
\cN_\delta(\Pi_{x_0} F_{x_0}) \gtrsim \delta^{-\alpha-\lambda/4}\quad\text{for all }x_0\in A'.
\ee

Now, since $R_\om^{(4\delta)}\subset \om^{(6\delta)}$, for each $\om\in\Omega'(x_0)$ there is another line $\om'$ parallel to $\om$ and at distance $\le 6\delta$ from it, such that
\[
\psi(\om') \cap (\{x_0\}\times\R) \cap Q \neq \varnothing.
\]
Then, since $\varphi=h\circ\psi$, we get from \eqref{eq:property-Pi} applied to $\ell=\psi(\om')$ that
\[
(x_0,\Pi_{x_0}(\varphi(\om')))\in Q \quad \text{for all } \om\in \Omega'(x_0).
\]
It follows from \eqref{eq:bi-Lipschitz} and $d(\om,\om')\le 6\delta$ that
\be \label{eq:almost-Fubini}
(x_0,\Pi_{x_0}(\varphi(\om))) \in Q^{(C \delta^{1-2\eta_3})} \quad \text{for all } \om\in \Omega'(x_0),
\ee
where $C>0$ is absolute. Since the set $A'$ is obtained by taking points from $\delta$-separated intervals, we deduce from \eqref{eq:size-A},  \eqref{eq:large-fiber-bis} and \eqref{eq:almost-Fubini} that
\[
\cN_\delta\left(Q^{(C \delta^{1-2\eta_3})}\right) \gtrsim |A'|\min_{x_0\in A'} \cN_\delta(\Pi_{x_0}F_{x_0})\gtrsim \delta^{-\alpha+6\gamma+\alpha \eta_3+3\e} \delta^{-\alpha-\lambda/4} .
\]
Since $Q^{(C \delta^{1-2\eta_3})}=Q+B(0,C\delta^{1-2\eta_3})$, it follows that
\[
\cN_\delta\left(Q^{(C \delta^{1-2\eta_3})}\right) \lesssim \cN_\delta(B(0,C\delta^{1-2\eta_3}))\cN_\delta(Q) \lesssim \delta^{-4\eta_3}\cN_\delta(Q).
\]
Combining the last two displayed equations, we get
\[
\cN_\delta(Q)\gtrsim \delta^{-2\alpha-\lambda/4+6\gamma+(4+\alpha) \eta_3+3\e} .
\]
Finally, by \eqref{eq:size-E-in-terms-of-Q}, this yields
\[
|E| \gtrsim \delta^{2-2\alpha-\lambda/4+6\gamma+(4+\alpha)\eta_3+3\eta_4+3\e}.
\]
If $\gamma_0=\gamma_0(\alpha)$ is small enough, then whenever $\gamma\le \gamma_0$ and $\e$ is sufficiently small, we have
\be \label{eq:final-condition-lambda}
6\gamma+(4+\alpha)\eta_3+3\eta_4+3\e \le \lambda/8.
\ee
Thus,
\[
|E|\gtrsim \delta^{2-2\alpha-\lambda/8},
\]
so we have gained $\lambda/8$ in the exponent; this number only depends on $\alpha$, and hence \eqref{eq:size-E} cannot hold if $\gamma$ is smaller than some $\gamma_0(\alpha)$ (smaller than $\lambda/8$, and small enough that all of the $\lambda_i$ are $<\lambda/4$, and that \eqref{eq:final-condition-lambda} holds). This is what we wanted to show.
\end{proof}

\begin{remark} \label{rem:quantitative}
Tracking the values of all the parameters $\eta_i$ and $\lambda_i$ (and letting $\e\to 0$) we see that in the end we obtain the condition
\[
\gamma_0 \ge \frac{\lambda(\alpha)}{176+656/\alpha},
\]
where $\lambda(\alpha)$ is the parameter from Theorem \ref{thm:bourgain} applied with $\beta=2\alpha$ and $\kappa=\alpha$. We have not tried to optimize this value in any way. Unfortunately, even though $\lambda$ is in principle computable from the existing proofs, no explicit estimate for it is known, and in any case it would extremely small.
\end{remark}

\appendix

\section{\texorpdfstring{The $\alpha+ \min\{ \beta, \alpha \}$-bound for $(\alpha, \beta)$-Furstenberg sets}{The alpha+min(alpha,beta) bound}}

\begin{theorem}
\label{lem:lutzstullapp}
Let $\alpha\in (0,1]$, $\beta\in (0,2n-2]$, and let $E \su \rr^n$ be an $(\alpha, \beta)$-Furstenberg set. Then $\hdim(E) \geq \alpha+ \min\{ \beta, \alpha \}$.
\end{theorem}

\begin{proof}
We may assume that $\beta\le \alpha$, since an $(\alpha,\beta)$-Furstenberg set contains an $(\alpha,\alpha)$-Furstenberg set if $\beta>\alpha$. We prove that every discretized
$(\delta,\alpha,\beta)$-Furstenberg set has measure $\gtrsim \delta^{n-(\al+\beta)}$. By Lemma \ref{lem:discrete-to-continuous}, this implies that every $(\alpha, \beta)$-Furstenberg set has Hausdorff dimension at least $\al+\beta=\al+\min(\al,\beta)$, and the statement follows.

Thus, let $E \su \rr^n$ be a discretized $(\delta,\alpha,\beta)$-Furstenberg set. That is,
$E=\cup_{\omega\in\Omega} R_\omega \subset B^n(0,2)$, where:
\begin{itemize}
\item $\Omega$ is $\delta$-separated and $\Omega^{(\delta)}$ is a $(\delta,\beta)$-set in $A(n,1)$.

\item For each $\omega\in\Omega$, $R_\omega$ is a $(\delta,\alpha)_n$-set contained in $\omega^{(2\delta)}$.

\item $|\Omega|\gtrsim \delta^{-\beta}$.
\end{itemize}

We use the following standard application of Cauchy-Schwarz.
\begin{lemma}
\label{lem:CSM}
Let $T_1,\dots,T_N \su \rr^n$ be measurable sets with finite Lebesgue measure. Then
$$\big| \bigcup_{j=1}^N T_j \big| \geq \frac{ \left( \sum_{j=1}^N |T_j| \right)^2}{\sum_{j=1}^N \sum_{i=1}^N |T_i \cap T_j |}.$$
\end{lemma}

\begin{proof}
Apply the Cauchy-Schwarz inequality to the functions $\sum_{j=1}^N \1_{T_j}$ and $\1_{\bigcup_{j=1}^N T_j}$.
\end{proof}

Our goal is to apply Lemma \ref{lem:CSM} to $\{R_{\om}\}_{\om \in \Om}$. Enumerate $\Om =\{\om_i\}_{i=1}^N$.
Since $N \approx \de^{-\beta}$, $|R_{\om_i}| \approx \de^{n-\al}$ for each $i$, we have
\begin{equation}
 \label{eq:sumofmeasures}
\left(\sum_{i=1}^N |R_{\om_i}| \right)^2 \approx \left(\de^{n-\al-\beta}\right)^2.
\end{equation}

Now we give an upper bound on the measure of the pairwise intersections of the $R_\om$'s. We use the following simple lemma about pairwise intersections of neighborhoods of lines.
\begin{lemma}
\label{lem:linedistances}
Let $\ell_1, \ell_2$ be two distinct lines in $\rr^n$. Then
$$\diam(\ell_1^{(2\de)} \cap \ell_2^{(2\de)} \cap B^n(0,2)) \lkb \frac{\de}{d(\ell_1,\ell_2)}.$$
\end{lemma}

\begin{proof}
Recall the definition of the metric $d$ on $A(n,1)$ from \S\ref{sec: metriclines}. We will use the following elementary results, see e.g. \cite{HKM19} and \cite{Wolff99} for reference.
\begin{lemma}
\label{lem:elementarylines}
Let $\ell_i = \langle e_i \rangle + v_i$, with $e_i\in S^{n-1}$, $v_i\in e_i^\perp$.
\begin{itemize}
	\item There exists a constant $C$ depending only on $n$ such that if $\ell_1, \ell_2 \in A(n,1)$ with $|v_1-v_2| > \angle(e_1,e_2) + C \de$, then
	$\ell_1^{(2\de)} \cap \ell_2^{(2\de)} \cap B^n(0,2) = \varnothing$.
	\item There exists a constant $C'$ depending only on $n$ such that if $\ell_1, \ell_2 \in A(n,1)$ with $\angle(e_1,e_2) > \de$, then
	$$\diam(\ell_1^{(2\de)} \cap \ell_2^{(2\de)} \cap B^n(0,2)) \leq \frac{C' \de}{\angle(\ell'_1,\ell'_2)}.$$
\end{itemize}
\end{lemma}

Let $\ell_1, \ell_2$ be two lines in $\rr^n$ with $d(\ell_1,\ell_2)>0$. Using Lemma \ref{lem:elementarylines},
the statement follows if $|v_1-v_2| > \angle(e_1,e_2) + C \de$. Assume now that $|v_1-v_2| \leq \angle(e_1,e_2) + C \de$.
Now, if $\angle(e_1,e_2) \leq \de$ then $d(\ell_1,\ell_2) \lkb \de$, thus
$$\diam(\ell_1^{(2\de)} \cap \ell_2^{(2\de)} \cap B^n(0,2)) \lkb 1 \lkb \frac{\de}{d(\ell_1,\ell_2)}.$$
On the other hand, if $\angle(e_1,e_2) > \de$, then $d(\ell_1,\ell_2) \lkb \angle(e_1,e_2)$, and thus
by Lemma \ref{lem:elementarylines}, we have
$$\diam(\ell_1^{(2\de)} \cap \ell_2^{(2\de)} \cap B^n(0,2)) \lkb \frac{\de}{\angle(e_1,e_2)} \lkb \frac{\de}{d(\ell_1,\ell_2)}.$$

\end{proof}

For any $i \neq j$, let $\ga_{i,j}$ denote the distance between the lines $\om_i, \om_j$; then $\ga_{i,j} \geq \de$ by assumption. Also, since all lines $\om_i$ intersect $B^n(0,2)$, the values of $\ga_{i,j}$ are bounded above by a constant $2^K$ (for example, $K=3$ works). By Lemma \ref{lem:linedistances}, $R_{\om_i} \cap R_{\om_j}$ is contained in the intersection of $R_{\om_i}$ and a ball of radius $\lkb \frac{\de}{\ga_{i,j}}$.
Using that $R_{\om_i}$ is a $(\delta,\alpha)_n$-set, this implies that
$$
|R_{\om_i} \cap R_{\om_j}| \lkb \de^n \ga_{i,j}^{-\al} \ \text{for any} \ i \neq j.
$$
Fixing $j$ and summing up for $i$ we obtain that
\begin{align*}
\sum_{i=1}^{N} |R_{\om_i} \cap R_{\om_j}|  & = \sum_{k=-K}^{\log{(1/\de)}} \sum_{\ga_{i,j} \in (2^{-k},2^{-k+1}]} |R_{\om_i} \cap R_{\om_j}| \lkb
\sum_{k=-K}^{\log{(1/\de)}} \sum_{\ga_{i,j} \in (2^{-k},2^{-k+1}]} \de^n \ga_{i,j}^{-\al} \\
& \lkb \sum_{k=-K}^{\log{(1/\de)}} |\{i: \ga_{i,j} \in (2^{-k},2^{-k+1}] \}| \de^n 2^{k\al}.
\end{align*}
The fact that $\Omega^{(\delta)}$ is a $(\delta,\beta)$-set in $A(n,1)$ easily implies that
$|\{i: \ga_{i,j} \in (2^{-k},2^{-k+1}] \}| \lkb 2^{-k\beta} \de^{-\beta}$.
Using this, we obtain that

$$\sum_{i=1}^{N} |R_{\om_i} \cap R_{\om_j}| \lkb \sum_{k=-K}^{\log{(1/\de)}} \de^{n-\beta} 2^{k(\al-\beta)}
  \lkb \de^{n-\beta} (1/\de)^{\al-\beta} =\de^{n-\al} ,
$$
using that $\beta\le \al$ and absorbing the $\log(1/\delta)$ factor into the $\lkb$ notation. Moreover, since $\Omega$ is $\delta$-separated and $\Omega^{(\delta)}$ is a $(\delta,\beta)$-set, we have $|\Om| \lkb \de^{-\beta}$.
Using this while summing up for $j$, we obtain
\begin{equation}
\label{eq:titjfinal}
\sum_{i,j=1}^{N} |R_{\om_i} \cap R_{\om_j}| \lkb \de^{n-\al-\beta}.
\end{equation}
Finally, combining \eqref{eq:sumofmeasures} and \eqref{eq:titjfinal}, we get
$$
|E| =\left| \bigcup_{j=1}^N R_{\om_j} \right| \gkb \frac{(\de^{n-\al-\beta})^2}{\de^{n-\al-\beta}}=\de^{n-(\al+\beta)},
$$
and the proof concludes.
\end{proof}


\end{document}